\documentclass[12pt]{amsart}

\usepackage[dvipsnames]{xcolor}
\usepackage{url}
\usepackage{graphicx}
\usepackage{amsmath}
\usepackage{color}
\usepackage{caption}
\usepackage{subcaption}
\usepackage{amsfonts,amssymb,amscd}

\pdfoutput=1

\calclayout

\newtheorem{thm}{Theorem}[section]
\newtheorem{prop}[thm]{Proposition}
\newtheorem{lem}[thm]{Lemma}
\newtheorem{cor}[thm]{Corollary}

\theoremstyle{definition}
\newtheorem{definition}[thm]{Definition}

\theoremstyle{remark}

\newtheorem{rem}[thm]{Remark}

\numberwithin{equation}{section}

\newcommand{\R}{\mathbb{R}}  
\newcommand{\D}{\mathbb{D}}
\newcommand{\C}{\mathbb{C}}  
\newcommand{\Z}{\mathbb{Z}}  
\newcommand{\N}{\mathbb{N}}  
\newcommand{\classB}{\mathcal{B}}
\newcommand{\Hr}{\mathbb{H}_r}

\newcommand{\scs}{\scriptsize}

\DeclareMathOperator{\interior}{int}

\DeclareMathOperator{\diam}{diam}

\DeclareMathOperator{\dist}{dist} 

\begin{document}


\title{Univalent wandering domains in the Eremenko-Lyubich class}


\author{N\'uria Fagella}
\address{Departament de Matem\`atiques i Inform\`atica, Institut de Matem\`atiques de la 
Universitat de Barcelona (IMUB) and Barcelona Graduate School of Mathematics (BGSMath). 
 Gran Via 585, 08007 Barcelona, Catalonia}
\email{nfagella@ub.edu}

\author{Xavier Jarque}
\address{Departament de Matem\`atiques i Inform\`atica, Institut de Matem\`atiques de la 
Universitat de Barcelona (IMUB), and Barcelona Graduate School of Mathematics (BGSMath).
 Gran Via 585, 08007 Barcelona, Catalonia}
\email{xavier.jarque@ub.edu}

\author{Kirill Lazebnik}
\address{California Institute of Technology. 1200 E California Blvd, Pasadena, CA 91125.}
\email{lazebnik@caltech.edu}

\thanks{ The first and second authors were partially supported by the Spanish 
grant MTM2017-86795-C3-3-P, the Maria de Maeztu Excellence Grant MDM-2014-0445, and grant 2017SGR1374 from the Generalitat de Catalunya. }





\begin{abstract}

We use the Folding Theorem of \cite{Bis15} to construct an entire function $f$ in class  $\mathcal{B}$ and a wandering domain $U$ of $f$ such that $f$ restricted to $f^n(U)$ is univalent, for all $n\geq 0$. The components of the wandering orbit are bounded and surrounded by the postcritical set.

\end{abstract}


\maketitle




\section{Introduction}


We consider the dynamical system formed by the iterates of an entire map $f:\C\to \C$. We will consider only \emph{transcendental} $f$, namely those maps $f$ with an essential singularity at $\infty$.  Such dynamical systems appear naturally as complexifications of one-dimensional real-analytic systems (interval maps or circle maps for instance), or as restrictions of analytic maps of $\R^{2n}$ to certain invariant one complex-dimensional manifolds. 


The dynamics of $f$ splits the complex plane into two complementary and  totally invariant sets: The {\em Fatou set} (or {\em stable set}), where the iterates form a normal family,  and its closed complement, the {\em Julia set}, $J(f)$, often a fractal  formed by chaotic orbits. 
The Fatou set is open and is generally composed of infinitely many connected components, known as {\em Fatou components}, which map among each other under the function $f$. 

It was already Fatou \cite{Fat20} who gave a complete classification of  periodic Fatou components in terms of the possible limit functions of the sequence of iterates. His classification theorem states that an invariant 
Fatou component is either an {\em immediate basin of attraction} of an attracting or parabolic fixed point; or a {\em Siegel disk}, i.e. a topological disk on which $f$ is conformally conjugate to a rigid irrational rotation;  or a {\em Baker domain} if the iterates converge uniformly to infinity. This classification extends to periodic Fatou components, since a component of period $p>1$ is invariant under $f^p$. 

It is well-known that each of the periodic cycles of Fatou components is in some sense associated to the orbit of a {\em singular value}, that is, a point around which not all branches of $f^{-1}$ are well defined. Singular values may be {\em critical values} (images of zeroes of $f'$) or also {\em asymptotic values} which, informally speaking,  are points that have at least one preimage ``at infinity'', such as $0$ for $z\mapsto e^z$.  

The set of singular values, $S(f)$,  plays a crucial role in holomorphic dynamics precisely because of its relation with the periodic Fatou components of $f$. We mention two examples. First, basins of attraction must  contain a singular value (and hence its forward orbit). Second, it is known that the {\em postsingular set}, $P(f)$,  i.e. the singular values of $f$ together with their forward orbits, must accumulate on the boundary of a Siegel disk \cite{Fat20,milnor}. The relation with Baker domains is weaker and not so easy to state, and we refer the reader to \cite{Berg95}. This connection allows one to glean information about possible stable orbit behaviours (namely periodic Fatou components) of a given map by understanding the dynamics of its set of singular values. For this reason, many classes of entire functions have been singled out in terms of properties of their singular set. Important examples are the {\em Speiser} class $\mathcal{S}$ of maps with a finite number of singular values (whose members are also called of {\em finite type}) or the {\em Eremenko-Lyubich} class
\[
\mathcal{B}=\{ f:\C\to \C  \text{\ entire} \mid S(f) \text{\ is bounded} \}.
\]
 
In the presence of an essential singularity at infinity there may exist Fatou components which are neither periodic nor eventually periodic. These are called {\em wandering domains} and are the subject of this paper. More precisely, a Fatou component $U$ is a wandering domain if $f^k(U) \cap f^j(U) =\emptyset$ for all $k,j\in \N$, $ k\neq j$.  

Perhaps because wandering domains do not exist for rational maps \cite{Sul85}, nor for maps in the Speiser class $\mathcal{S}$  \cite{EL92,GK86}, these rare Fatou components have not been subject of attention until quite recently, when maps with infinitely many singular values (like Newton's method applied to entire functions) have started to emerge as interesting objects. Nevertheless, many recent breakthrough results about wandering domains have appeared in the last several years. After the classical result \cite{EL92} which states that maps in class $\mathcal{B}$ cannot have wandering domains whose orbits converge to infinity uniformly (called {\em escaping wandering domains}), there was reasonable doubt of whether functions in class $\mathcal{B}$ could have wandering domains at all. This question was answered affirmatively by Bishop in \cite{Bis15} who constructed a function in class $\mathcal{B}$ with an {\em oscillating}  wandering domain, that is, a wandering domain whose orbits accumulate both at infinity and on a compact set (the collection of points that oscillate as such under $f$ is termed the \emph{Bungee set} $BU(f)$ - see \cite{OS16} for details and properties of this set). This construction depends on a technique, also developed in \cite{Bis15}, termed \emph{quasiconformal folding}. Very recently, Mart\'i-Pete and Shishikura \cite{2018arXiv180711820M} gave an alternative surgery construction of an $f\in\mathcal{B}$ with an oscillating wandering domain such that $f$ has finite order of growth. An oscillating wandering domain was constructed previously in \cite{EL87} using approximation theory techniques, however it is not known whether this example is in class $\mathcal{B}$. Indeed, these techniques give insufficient control over the singular set for this purpose. The question of existence of {\it dynamically bounded} wandering domains, i.e., wandering domains whose orbits do not accumulate at infinity, is unknown (this problem was first posed in \cite{EL87}). 

It is  a wide open problem to find the sharp relationship between wandering domains and the postsingular set $P(f)$, or even with  the singular set $S(f)$.  Results up to now show that some relation exists: if the domain is oscillating (i.e. lies inside $BU(f)$), any finite limit function must be a constant in $J(f)\cap \overline{P(f)}$ \cite{Bak02} (see also \cite{Haruta}) and, in any case, there must be postsingular points inside or nearby the wandering components (see \cite{BFJK17} and \cite{MR13} for the precise statements). 

A very related and natural  question is whether a wandering domain could exist such that the function were univalent on each of the orbit components.  Outside the class $\mathcal{B}$ the answer is, not surprisingly, affirmative, as shown in \cite{EL87} and \cite[Example 1]{FH09}. The example of \cite{EL87} is obtained using approximation theory, whereas the escaping wandering domain of \cite[Example 1]{FH09} is a logarithmic  lift of an appropriately  chosen invariant Siegel disk. But it is inside class $\mathcal{B}$ where this question makes most sense to be asked.  As we mention above, wandering domains in class $\mathcal{B}$ can not be escaping, and hence a large amount of contraction is necessary.  Let us keep in mind that Bishop's example contains a critical point of very high order inside infinitely many of its components, which allows for this large contraction. Our main result in this paper shows that, nevertheless,  univalent wandering domains are also possible inside this class of maps. 

\begin{thm}\label{mainthm}
There exists an entire transcendental function $f \in \mathcal{B}$  and a wandering Fatou component $U$ of $f$ such that  $f|_{f^n(U)}$ is univalent for all $n\geq 0$.
\end{thm}

Our example uses the Folding Theorem in \cite{Bis15} (see Section 2) and is in fact a careful modification of Bishop's original construction. Very roughly speaking Bishop's function behaves like $(z-z_n)^{m_n}$ for some $m_n\to\infty$, on some subsequence of  wandering components. We replace these maps by $(z-z_n)^{m_n} + \delta_n\cdot(z-z_n)$ on subsets of the same components, which are univalent near the points $z_n$, and show that the critical values  can be kept outside (but very close to) the actual wandering components. This is achieved by trapping the boundary of the wandering components inside annuli of decreasing moduli which separate the domains from the critical values. 

The present construction yields rather poor control of the singular orbits in comparison to the initial construction of \cite{Bis15}. Indeed in Bishop's construction of a wandering domain in class $\mathcal{B}$, the singular orbits are well understood: a singular orbit either lies inside the wandering component or on the real line. Using this fact, it is proven in \cite{FGJ15} that there are no other wandering components in the construction of \cite{Bis15} apart from those explicitly constructed. In the construction of the present paper, however, the singular orbits are not well understood apart from the guarantee that the critical points can not lie in the forward orbit of the wandering component (see Section \ref{univalence}). Thus it is left as a possibility that the {\it grand} orbit of the wandering component may contain a critical point. Indeed we also leave open the question of whether there are {\it other} wandering components of the present example besides those explicitly constructed. Lastly we ask whether one can arrange for finer control over the singular orbits: is it possible to arrange for the singular set to lie entirely inside the Fatou or Julia sets? 

%

This paper is organized as follows. In Section 2 we recall Bishop's construction of a wandering domain in class $\mathcal{B}$, together with other preliminary results we use later on in the proof of Theorem \ref{mainthm}. Section 3 describes the new map on $D$-components that we shall work with in the construction. Section 4 describes a parametrized family of entire functions, from which we select in Section 5 a particular choice of parameters corresponding to an entire function possessing a wandering component. In Section \ref{univalence}, it is proven that $f$ acts univalently on the forward orbit of this wandering component, thus finishing the proof of Theorem \ref{mainthm}. 


\subsection*{Acknowledgements}
We are indebted to David Mart\'i-Pete and Mitsuhiro Shishikura for their careful reading of and numerous comments on a previous version of this paper. We are grateful to Chris Bishop for his comments on a preliminary version of the paper. We would also like to thank Mikhail Lyubich and Lasse Rempe-Gillen for helpful discussions. We finally thank the Universitat de Barcelona and the  Institut de Matem\`atiques de la UB for their hospitality during the visits that led to this work.


\section{Preliminaries}

The main result of this paper is based on Bishop's construction of transcendental entire maps via quasiconformal folding  \cite[Theorem 1.1]{Bis15}. Roughly speaking, given an infinite bipartite graph $T$, Bishop's Folding Theorem provides an entire function in class $\classB$ (bounded singular set) with  prescribed behaviour (up to pre-composition with a quasiconformal map close to the identity) off a small neighborhood of $T$. As opposed to what occurs with other existence theorems  (for example those in approximation theory - see for instance \cite{Gai}), one has fine control of the singular set of the final map, which makes this tool effective when constructing  examples in restricted classes of functions such as class $\mathcal{B}$.  

Our goal in this section is to provide the reader with the essential background to state (a simplified version of) Bishop's Theorem (Subsection 2.1), and its application to produce a transcendental entire function in class $\classB$ having an oscillating wandering domain (Subsection 2.2). For a deeper discussion we refer to the source \cite{Bis15} or to \cite{FGJ15, Laz} where some details are  more explicit. Additionally, at the end of this section we recall  some additional tools that will be used throughout the paper.

\subsection{On Bishop's quasiconformal folding construction}
Let $T$ be an unbounded connected bipartite graph with vertex labels in $\{-1,+1\}$. Then the connected components of $\C\setminus T$ are simply connected domains in $\C$. We denote by $R$-{\it components} (respectively  $D$-{\it components}) the unbounded (respectively bounded) components of $\C\setminus T$. We will assume that $T$ has \emph{uniformly bounded geometry}, i.e. that edges are (uniformly) $\mathcal{C}^{2}$ and that the diameters of edges satisfy certain uniform bounds (see Theorem 1.1 of \cite{Bis15}, the note \cite{Bishop_correction_note}, and Section 2 of \cite{2018arXiv180704581B}).
We define a neighbourhood of the graph given by
$$
T(r):=\bigcup_{e\text{ edge of }T}\Big\{z\in\C\ |\ \dist(z,e)<r\diam(e)\Big\},
$$
where $\dist$ and $\diam$ denote the Euclidean distance and diameter respectively. 

We denote by $\Hr=\{z=x+iy\in\C\,|\,x>0\}$ the right half plane and by $\D=\{z\in\C\,|\,|z|<1\}$ the unit disk. For each connected component $\Omega_{j}$ of $\C\setminus T$, let $\tau_{j}:\Omega_{j}\to \Delta_j$ be the Riemann map where $\Delta_j=\Hr$ or $\Delta_j=\D$ depending on whether $\Omega_{j}$ is an unbounded or bounded component. We call $\Delta=\Delta_j$  the {\em standard domain} for $\Omega_{j}$. We shall denote by $\tau$ the global map defined on $\cup_j \Omega_j$ such that $\tau|_{\Omega_j} = \tau_j$. Each edge $e$ of $T$ is the common boundary of at most two complementary domains but corresponds via $\tau$ to exactly two intervals on $\partial \Hr$, or one interval on $\partial \Hr$ and one arc in $\partial \D$ (see condition (i) in Theorem \ref{folding}). The \textit{$\tau$-size} of an edge $e$ is defined to be the minimum among the lengths of the two images of $e$ under $\tau$.   

Moreover we also define a map $\sigma$ from the standard domains $\Delta_j$ to $\mathbb C$ depending on whether  $\Delta_j$ equals $\mathbb H_r$ or $\mathbb D$. More precisely, we define $\sigma(z):=\exp(z)$ if 
$\Delta_j=\mathbb H_r$. Otherwise, if  $\Delta_j=\mathbb D$, then $\sigma(z):=z^m,\ m\geq 2$ possibly followed by a quasiconformal map $\rho: \mathbb D \mapsto \mathbb D$ which sends $0$ to some $w\in \interior (\mathbb D)$ and is the identity on $\partial \mathbb D$.  

Now we are ready to state Theorem 7.2 of \cite{Bis15}, in a simplified version which is sufficient for our purposes.

\begin{thm}\label{folding} Let $T$ be an unbounded connected graph and let $\tau$ be a conformal map defined on each complementary domain $\mathbb{C}\setminus T$ as above. Assume that:

\begin{itemize}
\item[(i)] No two D-components of $\mathbb{C}\setminus T$ share a common edge.
\item[(ii)] T is bipartite with uniformly bounded geometry.
\item[(iii)] The map $\tau$ on a D-component with $2n$ edges maps the vertices to the $2n^{\text{th}}$ roots of unity.
\item[(iv)] On R-components the $\tau$-sizes of all edges are uniformly bounded from below.
\end{itemize}
\vspace{2mm}

Then  there is an $r_0>0$, a transcendental entire $f$, and a $K$-quasiconformal map $\phi$ of the plane, with $K$ depending only on the uniformly bounded geometry constants, so that $f=\sigma\circ\tau\circ\phi^{-1}$ off $T(r_0)$. Moreover $f$, has no asymptotic values, and the only critical values of $f$ are $\pm1$ and those critical values assigned by the D-components.
\end{thm}
 
 As mentioned, the proof of the above Theorem is based on some quasi-conformal deformations of the  
maps $\tau$ and $\sigma$ inside $T_{r_0}$ so that the modified global map $g=\sigma \circ \tau$ is quasiregular. In this situation the Measurable Riemann Mapping Theorem (see Theorem \ref{MRT} below) provides the quasiconformal map $\phi$ in the statement of Theorem \ref{folding}. We sketch here some details needed for our construction but we refer again to \cite{Bis15}, \cite{Bishop_correction_note} and Sections 2, 3 of \cite{2018arXiv180704581B} for a more detailed approach.

If $\Delta=\Hr$ we first divide $\partial\Hr$ into intervals $I$ of length $2\pi$ and vertices in $2\pi i\Z$. These intervals will correspond, after Bishop's folding construction, to the images of the edges of $T'\supset T$ (where $T'$ is $T$ with some \textit{decorations} added, all of which are contained  in the neighborhood $T(r_{0})$ of $T$) by a suitable quasiconformal deformation $\eta$ of $\tau$. Secondly we define $\sigma$ on $\partial\Hr$. There are two cases to consider: either $I$ is identified with a common arc of two $R$-components or $I$ is identified with a common arc of one $R$-component and one $D$-component. In the latter case, we define $\sigma(iy):=\exp(iy)$ for every $iy\in I$. In the former case, we define $\sigma(iy):=M\circ\exp(iy)$ or $\sigma(iy):=M^{-1}\circ\exp(iy)$, for every $iy\in I$, where $M(z)=i(z-i)/(z+i)$ is a M\"obius transformation sending the upper half of the unit circle to $[-1,1]$ (and $M^{-1}$ sends the lower half of the unit circle to $[-1,1]$), and we use $M$ or $M^{-1}$ depending on whether $\exp(iy)$ maps to the upper or lower half of the unit-circle for $iy\in I$. Lastly, we extend $\sigma$ to $\mathbb{H}_r$ as a quasiregular map such that $\sigma(z)=\exp(z)$ for $\textrm{Re}(z)>2\pi$ (see the note \cite{Bishop_correction_note} or Section 3 of \cite{2018arXiv180704581B}).


\subsection{The prototype map} \label{subsection:prototype} In \cite[Section 17]{Bis15}, the author gives an application of Theorem \ref{folding} in order to construct a family of entire functions in class $\classB$ depending on infinitely many parameters. One defines an unbounded connected graph $T$ and, on some relevant  complementary domains, defines the maps $\sigma\circ\tau$ depending on the parameters. By choosing the parameters appropriately, it is ensured that the resulting function has oscillating wandering domains. Since Theorem \ref{mainthm} is also an application of Theorem \ref{folding} to the same family of graphs as in \cite[Section 17]{Bis15}, we briefly describe the construction. Again we refer to \cite{Bis15} or \cite{FGJ15,Laz}  for a detailed discussion.  

Consider the open half strip
$$
S^{+}:=\left\{x+iy\in\C\ |\ x>0\ \text{and}\ |y|<\frac{\pi}{2}\right\}.
$$
Following the previous notation we denote by $(\sigma\circ\tau)|_{S^{+}}$ the composition $z\mapsto\sigma(\lambda\sinh(z))$, where we remark that $S^+$ will neighbor only R-components so that $\sigma$ is determined as in the last paragraph of Section 2.1. We remark that this map extends continuously  to the boundary, sending $\partial S^{+}$ onto the real segment $[-1,+1]$. On the upper horizontal boundary of $ S^{+}$, we select points $(a_{n}\pm i\pi/2)_{n\geqslant 1}$ which are sent to $\{-1,+1\}$ by $(\sigma\circ\tau)|_{S^{+}}$, such that $a_{n}$ is close to $n\pi$ for every $n\geqslant 1$ (see \cite{FGJ15} for more details).

The following open disks will belong to the graph $T$  (see Figure \ref{fig:domains}):
$$
\forall n\geqslant 1,\quad D_{n}:=\{z\in\C\ |\ |z-z_{n}|<1\}\quad\text{where}\quad z_{n}:=a_n+i\pi.
$$
We complete the construction of $T$ by adding segments connecting the points $a_n+i\pi/2$ and $z_n-i$, adding vertical segments connecting the points $z_n+i$ with infinity, and lastly, copying the structure through the symmetries $z \to \pm \overline{z}$.

Again, following the above notation, we will denote by $(\sigma\circ\tau)|_{D_{n}}$ the composition $z\mapsto\rho_{n}(\left(z-z_{n})^{m_{n}}\right)$ for every $n\geqslant 1$ where, for every $n$,  $m_n\in2\mathbb{N}$ and $\rho_n$ is a quasiconformal map sending $0$ to $w_n\in D(0,3/4)$ (see Lemma \ref{rho}) and such that $\rho_n|_{\partial \mathbb D}={\rm Id}$.  Figure \ref{fig:domains} summarizes the construction.

For suitable choices of the parameters $\{w_n,m_n,\lambda\}$, Theorem \ref{folding} gives an example of a transcendental map in class $\classB$ with an oscillating (non-univalent) wandering domain (\cite{Bis15}, \cite{Bishop_correction_note}). The contribution of the present work is to show that by considering a more general class of quasiregular maps on $D$-components (see Section 3), one is able to construct a function in class $\mathcal{B}$ which is then proven to have a \emph{univalent} wandering domain.


\begin{figure}[!htbp]
\centering
\setlength{\unitlength}{0.9\textwidth}
\includegraphics[width=0.9\textwidth]{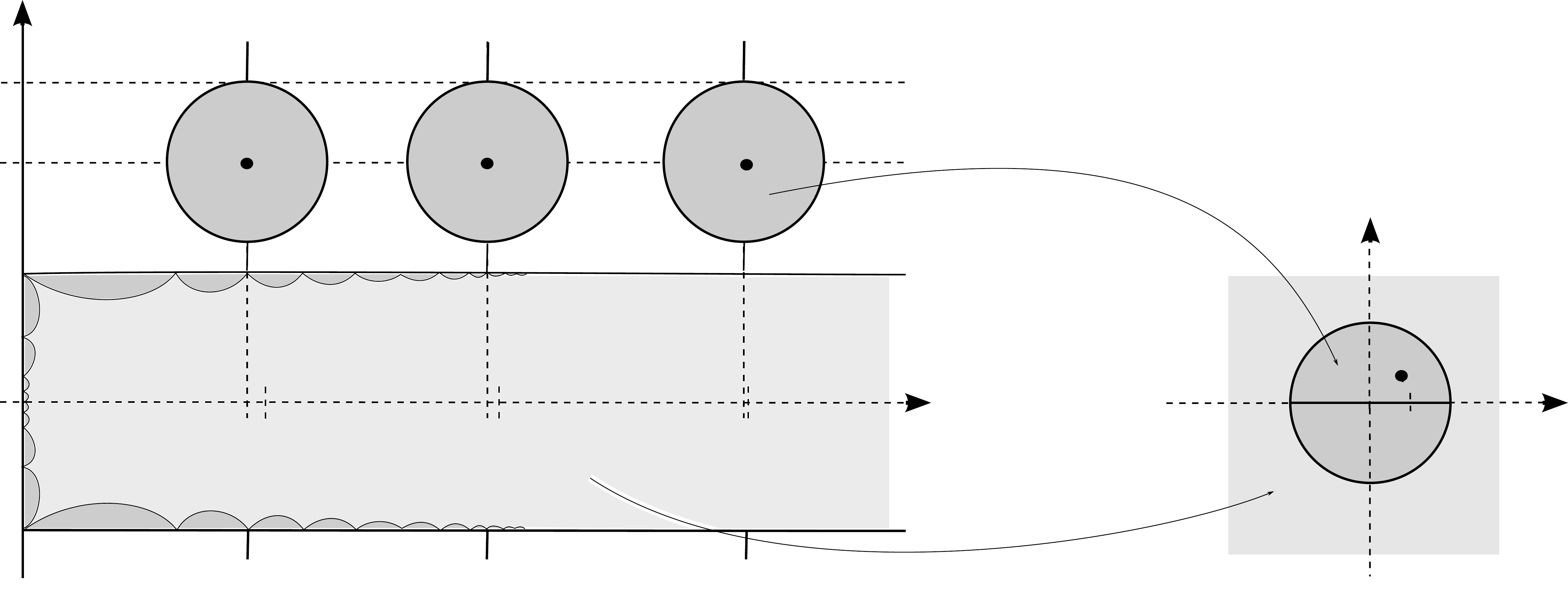}
    \put(-1.07,0.32522965){\scs$i(\pi+1)$}%
    \put(-1.02644847,0.27460868){\scs$i\pi$}%
    \put(-1.02897605,0.20816413){\scs$i\frac{\pi}{2}$}%
    \put(-1.01066558,0.1200105){\scs$0$}%
    \put(-0.91889615,0.14677866){$S^+$}%
    \put(-0.87076763,0.30021838){$D_1$}%
    \put(-0.84896265,0.25530134){\scs$z_1$}%
    \put(-0.71584805,0.30021838){$D_2$}%
    \put(-0.70398949,0.25530134){\scs$z_2$}%
    \put(-0.54010956,0.30021838){$D_3$}%
    \put(-0.52892506,0.25530134){\scs$z_3$}%
    \put(-0.85896265,0.1012687){\scs$a_1$}%
    \put(-0.8374386,0.1012687){\scs$\pi$}%
    \put(-0.71,0.1012687){\scs$a_2$}%
    \put(-0.69,0.1012687){\scs$2\pi$}%
    \put(-0.55,0.1012687){\scs$a_3$}%
    \put(-0.40868932,0.05010627){\scs $\sigma(\lambda\sinh(z))$}%
    \put(-0.4,0.2887626){\scs $\rho_n\left((z-z_n)^{m_n}\right)$}%
    \put(-0.21,0.1012687){\scs$-1$}%
    \put(-0.07,0.1012687){\scs$1$}%
    \put(-0.14,0.1012687){\scs$0$}%
    \put(-0.11,0.1012687){$\frac12$}%
    \put(-0.11,0.1512687){\scs$w_n$}%
    \caption{\small The domains $S^{+}$ and $(D_{n})_{n\geqslant 1}$ are depicted on the left. The dark gray areas represent the preimages of the unit disk $\D$ under the map $\sigma\circ\tau$.}\label{fig:domains}
\end{figure}
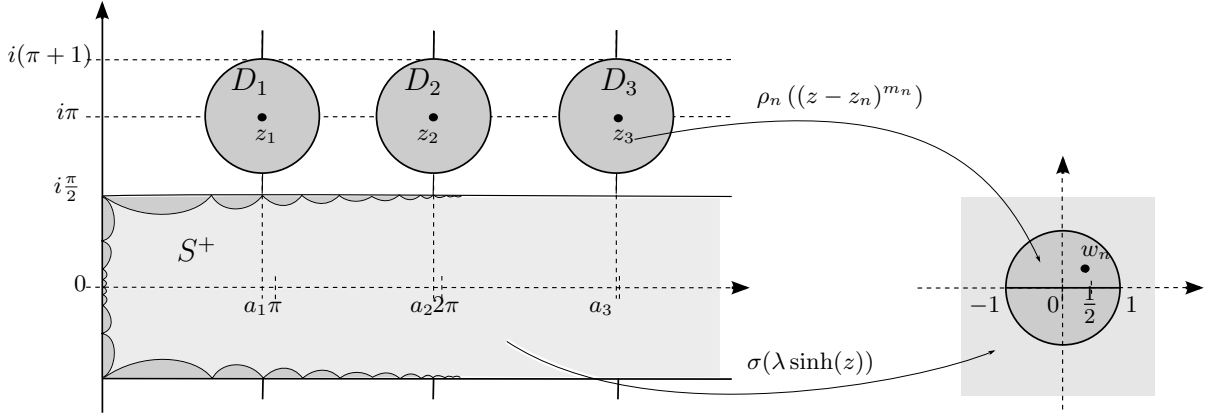

\subsection{Other tools} The following statement is  Koebe's one-quarter Theorem and a part of his distortion theorem.
\begin{thm}[{\cite[Section 1.3]{PommerenkeBook}}]\label{thm:Koebe}
Let $F$ be a univalent function on the disk $D(a,r)$ for some $a\in\C$ and $r>0$.  Then
\begin{itemize}
\item[(a)] $F(D(a,r)) \supset D\left( F(a),\frac14 |F'(a)| r\right)$. 

\item[(b)] For all $z\in D(a,r)$.

\[
 \dfrac{r^2|z-a||F'(a)|}{(r+|z-a|)^2}\leqslant |F(z)-F(a)| \leqslant\dfrac{r^2|z-a||F'(a)|}{(r-|z-a|)^2}.
 \]
 \item[(c)]  For all $z\in D(a,r)$,

\[
 \dfrac{1-\left| \frac{z-a}{r} \right|}{(1+|\frac{z-a}{r}|)^3}\leqslant \left| \frac{F'(z)}{F'(a)} \right| \leqslant\dfrac{1+\left| \frac{z-a}{r} \right|}{(1-|\frac{z-a}{r}|)^3}.
 \]

\end{itemize}
\end{thm}

Theorem \ref{MRT} below is termed the Measurable Riemann Mapping Theorem. For a proof, history, and references we refer to Chapter 4 of \cite{MR2245223}. It is used to produce the quasiconformal mapping $\phi$ of Theorem \ref{folding}. Theorem \ref{Lehto-Virtanen} below will be used in normal family arguments to deduce the dilatation of a limit of a sequence of quasiconformal mappings. An exposition of Theorem \ref{Lehto-Virtanen} is given in Section IV.5.6 of \cite{MR0344463}. 

\begin{thm}{\emph{}}
\label{MRT}
If $\mu\in L^{\infty}(\mathbb{C})$ with $||\mu||_\infty<1$, there exists a quasiconformal mapping $\phi:\mathbb{C}\rightarrow\mathbb{C}$ so that $\phi_{\overline{z}}/\phi_z=\mu$ a.e.. Moreover, given any other quasiconformal $\Phi:\mathbb{C}\rightarrow\mathbb{C}$ with $\phi_{\overline{z}}/\phi_z=\Phi_{\overline{z}}/\Phi_z$ a.e., there exists a conformal $\psi: \mathbb{C}\rightarrow\mathbb{C}$ so that $\Phi= \psi\circ\phi$.  

\end{thm}

\begin{thm}{\emph{(\cite{MR0083025})}} \label{Lehto-Virtanen} Let $\phi_n:\mathbb{C}\rightarrow\mathbb{C}$ be a sequence of $K$-quasiconformal mappings converging to a quasiconformal mapping $\phi: \mathbb{C}\rightarrow\mathbb{C}$ with complex dilatation $\mu$ uniformly on compact subsets of $\mathbb{C}$. If the complex dilatations $\mu_n(z)$ of $\phi_n$ tend to a limit $\mu_\infty(z)$ almost everywhere, then $\mu_\infty(z)=\mu(z)$ almost everywhere.  

\end{thm}

%
%
%
%
%
%

As explained above, the key idea behind Theorem \ref{mainthm} is to obtain the desired entire function $f$ as the composition of a quasiregular map $\sigma\circ\eta$ as given by Theorem \ref{folding}, and a quasiconformal map $\phi$ given by Theorem \ref{MRT}, that is $f:=(\sigma\circ\eta)\circ\phi^{-1}$. In particular, $f$ and $\sigma\circ\eta$ are not conjugate to each other. As it turns out, we shall have an explicit expression for $\sigma\circ\eta$,  at least in the domains where the relevant dynamics occur. In order then to control the  dynamics of $f$ one needs control on the correction map $\phi$. This will be a consequence of a uniform bound on the dilatation of $\sigma\circ\eta$ which is essentially independent of the parameters (see Theorem \ref{folding}), and the vanishing of the support of the dilatation of $\sigma\circ\eta$ for increasing parameters. Here we will appeal to a result of Dyn'kin \cite{MR1466801} about conformality of a quasiconformal mapping at a point, and some basic results about the normality of a family of normalized $K$-quasiconformal mappings (see, for instance, Section II.5 of \cite{MR0344463}). 

\section{The map on $D$-components}
\label{interpolation}

In this Section we describe two quasiregular self-maps $\psi$, $\rho$ of the unit disc. Let us consider the $C^{\infty}$ bump function

\[ b(x)=\begin{cases} 
      \exp(1+\frac{1}{x^2-1}) & \textrm{ if } 0\leq x < 1 \\
      0 & \textrm{ if } x\geq 1,
   \end{cases}
\]
define the transformation $\phi(x):=\frac{x-r}{1-r}$, and the smooth map 
\[ \hat{\eta}(x)=\begin{cases} 
      1 & \textrm{ if }x\leq r \\
      b(\phi(x)) & \textrm{ if } r \leq x \leq 1 \\
      0 & \textrm{ if }x\geq 1.
   \end{cases}
\]
We also set $\eta(z)=\hat{\eta}\left(|z|\right)$.

\begin{lem}\label{interpolation_map} Let $\psi(z):= z^m + \delta z \eta(z)$ for $z\in\mathbb{D}$ with $r:= 1-(4\delta)/m$. There exist $m_0\in\mathbb{N}$, and $\delta_0>0$ such that if $m>m_0$ and $\delta<\delta_0$, then $r>(\delta/m)^{1/(m-1)}$ and $||\frac{\psi_{\overline{z}}}{\psi_z}||_{L^\infty(\mathbb{D})}<4/5$. 

\end{lem}

\begin{proof} Using the chain rule, we compute (for any $m\in\mathbb{N}$ and $\delta>0$):

\[ \psi_z(z)=mz^{m-1}+\delta\eta(z)+\delta z\eta_z(z)\quad\text{and }\quad  \psi_{\overline{z}}(z)=\delta z\eta_{\overline{z}}(z). \]

\vspace{2mm}

\noindent Note that $|\hat{\eta}'(x)|=|b'(\phi(x))|\cdot|\phi'(x)|$ on the relevant interval. Solving $b''(x)=0$, one sees that $|b'(x)|$ has a maximum at $x_0=(1/3)^{1/4}$ with $|b'(x_0)|<e$, so that $|b'(x)|\leq e$ for $x\in[0,1]$. Since $\phi'(x)=1/(1-r)$, we have that $|\hat{\eta}'(x)|\leq e/(1-r)$ for all $x>0$. Using the chain rule again, we have 

\[ \left| \frac{\partial\eta}{\partial z}(z) \right| = \left| \hat{\eta}'(|z|) \right|\cdot\left| \frac{\partial|z|}{\partial z} \right| \leq \frac{e}{2(1-r)}\quad \textrm{and} \quad \left| \frac{\partial\eta}{\partial \overline{z}}(z) \right| = \left| \hat{\eta}'(|z|) \right|\cdot\left| \frac{\partial|z|}{\partial \overline{z}} \right| \leq \frac{e}{2(1-r)},\]

\vspace{2mm}

\noindent where we have used the fact that 
$$\frac{\partial|z|}{\partial z}=\frac{\overline{z}}{2|z|} \quad {\rm and} \quad \frac{\partial|z|}{\partial \overline{z}}=\frac{z}{2|z|}.$$
Hence

\begin{equation}\label{ineq1} \left|\frac{\psi_{\overline{z}}(z)}{\psi_z(z)}\right| \leq \frac{\delta|z|\frac{e}{2(1-r)}}{\left| m|z|^{m-1} - \delta|\eta(z)| - \delta|z|\frac{e}{2(1-r)} \right| }.  \end{equation}


\noindent We show that since we have chosen $r:= 1-(4\delta)/m$, there is some $M_1\in\mathbb{N}$ so that for $m\geq M_1$ one indeed has $r>(\delta/m)^{1/(m-1)}$ for any $0<\delta<1$. First note that the inequality $r>(\delta/m)^{1/(m-1)}$ can be rearranged to $m(1-(\delta/m)^{1/(m-1)})>4\delta$. Since $\delta < 1 $ it is clear that $m(1-(\delta/m)^{1/(m-1)})>m(1-(1/m)^{1/(m-1)})$, and so it will suffice to show that $m(1-(1/m)^{1/(m-1)})\rightarrow\infty$ as $m\rightarrow\infty$, and this is a simple calculation done by applying L'Hopital's rule to the quotient $(1-(1/m)^{1/(m-1)})/(1/m)$. Since we have fixed $r=1-\frac{4\delta}{m}$, we have 
$e/(2(1-r))=em/(8\delta)$, and so (\ref{ineq1}) turns to be

\begin{equation}
\label{ineq2} \left|\frac{\psi_{\overline{z}}(z)}{\psi_z(z)}\right| \leq \frac{\delta|z|\frac{e}{2(1-r)}}{\left| m|z|^{m-1} - \delta|\eta(z)| - \delta|z|\frac{e}{2(1-r)} \right| }   \leq \frac{e/8}{\left||z|^{m-2} - \frac{\delta|\eta(z)|}{m|z|} - \frac{e}{8} \right|}.\end{equation}

\noindent Moreover, for $|z|>r$ (where $\psi_{\overline{z}}(z)\not=0$), we have

\[ |z|^{m-2} - \frac{\delta|\eta(z)|}{m|z|} - \frac{e}{8} \geq \left(1-\frac{4\delta}{m}\right)^{m-2} - \frac{\delta}{m-4\delta} - \frac{e}{8}, \]

\noindent and 
$$
\left(1-\frac{4\delta}{m}\right)^{m-2} \rightarrow e^{-4\delta}\quad {\rm and} \quad  \delta/(m-4\delta) \rightarrow 0, \quad  {\rm as} \ \ m\rightarrow\infty.
$$  
Let $\delta_0$ be such that for $\delta<\delta_0$, $(8\delta)/(e(1-4\delta))<1/8$ and $8e^{-4\delta-1}>5/2$, where we note that $8/e > 5/2$. Hence since $(8\delta)/(e(m-4\delta))< (8\delta)/(e(1-4\delta))$, we have $(8\delta)/(e(m-4\delta))<1/8$ for $\delta<\delta_0$ and all $m\in\mathbb{N}$. Let $m_0\geq M_1$ be such that for $m\geq m_0$ and $\delta<\delta_0$, $\left(1-\frac{4\delta}{m}\right)^{m-2} > e^{-4\delta}-e/(64)$. Then for $m \geq m_0$ and $\delta < \delta_0$, 

\begin{equation*}
\begin{split} 
\left(1-\frac{4\delta}{m}\right)^{m-2} - \frac{\delta}{m-4\delta} - \frac{e}{8}  &\geq e^{-4\delta} - \frac{e}{64} - \frac{e}{64} - \frac{e}{8} \\
&= \left(e/8\right)\left(8e^{-1-4\delta} - 1/4 - 1\right) > \left(e/8\right)5/4. 
\end{split}
\end{equation*}

Finally we conclude from (\ref{ineq2}) that

\[\left|\frac{\psi_{\overline{z}}(z)}{\psi_z(z)}\right| \leq  \frac{e/8}{\left||z|^{m-2} - \frac{\delta|\eta(z)|}{m|z|} - \frac{e}{8} \right|} \leq \frac{e/8}{\left( \left(1-\frac{4\delta}{m}\right)^{m-2} - \frac{\delta}{m-4\delta} - \frac{e}{8} \right)} \leq 4/5. \]

\end{proof}

The map $\psi$ defined in Lemma \ref{interpolation_map} is an interpolation between $z\rightarrow z^m+\delta z$ in a subdisc of $\mathbb{D}$ and $z\rightarrow z^m$ on $\partial\mathbb{D}$. We will use the notation $\psi_{\delta, m}$ when we wish to emphasize the dependence of the map $\psi$ on the parameters $\delta>0$ and $m\in\mathbb{N}$. The $m-1$ critical points $(c_k)$ and $m-1$ critical values $(v_k)$ of the map $\psi$ are given by
$$
c_k=\left(\frac{-\delta}{m}\right)^{\left(\frac{1}{m-1}\right)}, \quad v_k=\delta\left(\frac{-\delta}{m}\right)^{\left(\frac{1}{m-1}\right)}\left(\frac{m-1}{m}\right).
$$
If we fix $\delta$ and let $m\rightarrow\infty$, the critical points will tend to $\partial\mathbb{D}$, while the critical values tend, in modulus, to $\delta$. 

Next we recall, from \cite{Bis15}, the definition of a quasiconformal map $\rho_w:\mathbb{D}\rightarrow\mathbb{D}$ which is the identity on $|z|=1$, conformal on a region containing $0$, and perturbs the origin to $w$:

 \[ \rho_w(z)=\begin{cases} 
      z+w & \textrm{ if } 0\leq|z|\leq1/8 \\
      z\frac{(8|z|-1)}{7}+(z +w)\frac{8-8|z|}{7} & \textrm{ if } 1/8\leq|z|\leq 1.
   \end{cases}
\]

\begin{lem}
\label{rho}

There exists a constant $k_0<1$ independent of $w\in D(0,3/4)$ such that $||\frac{(\rho_w)_{\overline{z}}}{(\rho_w)_z}||_{L^\infty(\mathbb{D})}<k_0$. 

\end{lem}

\noindent For the proof of Lemma \ref{rho}, see Section 3 of \cite{FGJ15}. In the following sections, we will consider the following quasiregular self-map of the unit disc:

\begin{equation}\label{composition} \rho_w \circ \psi_{\delta, m}: \mathbb{D} \rightarrow\mathbb{D}.
\end{equation}

\begin{prop}\label{small_dilatation} Let $s<1$, and let $\mu^{w, \delta, m}:=(\rho_w \circ \psi_{\delta, m})_{\overline{z}}/(\rho_w \circ \psi_{\delta, m})_z$. There exists $m_0 \in \mathbb{N}$ (depending on $s$) and $\delta_0$, $k_0$ (independent of $s$), such that if $m\geq m_0$, $\delta<\delta_0$, and $w\in D(0,3/4)$, one has $||\mu^{w, \delta, m}||_{L^{\infty}}<k_0$ and  $\textrm{supp}(\mu^{w, \delta, m}) \subset \{ z \in \mathbb{D} : |z|>s\}$.

\end{prop}

\begin{proof} Fix $s<1$, and let $\delta_0$ be as in Lemma \ref{interpolation_map}, with the extra condition that $\delta_0<1/16$. Let $k_0$ be as in Lemma \ref{rho}, ensuring also that $k_0<4/5$. The statement $||\iota^{w,\delta,m}_{\overline{z}}/\iota^{w,\delta,m}_{z}||_{L^{\infty}}<k_0$ then follows. 

Since $r:=1-4\delta/m>1-4/m\rightarrow1$ as $m\rightarrow\infty$ and $\psi_{\delta, m}$ is holomorphic on $r\mathbb{D}$, it follows that $\textrm{supp}((\psi_{\delta, m})_{\overline{z}}/(\psi_{\delta, m})_{z})\subset  \{ z \in \mathbb{D} : |z|>s\}$ for large $m$. It remains to consider the pullback of $\textrm{supp}(\rho_w)=\{ z\in\mathbb{D} : |z| \geq 1/8 \}$ under $\psi_{\delta, m}$. Since $\delta<\delta_0<1/16$, for large $m$ and $|z|<s$ we have that $s^m<1/16$, so that $|z^m+\delta z|<1/8$, and hence the pullback of $\textrm{supp}(\rho_w)=\{ z\in\mathbb{D} : |z| \geq 1/8 \}$ under $\psi_{\delta, m}$ is contained in $s<|z|\leq1$. 

\end{proof}


\section{A Base family of quasiregular maps} 

In this Section we define a family of quasiregular maps $g$ depending on several sets of parameters from which we will, in Section 5, choose one such $g$ for which $g\circ\phi^{-1}$ has a wandering domain, where $\phi$ is a quasiconformal map from Theorem \ref{MRT} such that $g\circ\phi^{-1}$ is holomorphic. First we define $g$ in a subset of the plane:

\begin{equation} \label{quasiregular_definition} g(z)=\begin{cases} 
      \sigma(\lambda\sinh(z)) & \textrm{ if } z\in S^+ \\
      \rho_{w_n} \circ \psi_{\delta_n, m_n} \circ (z\rightarrow z-z_n) & \textrm{ if } z \in D_n, \\
   \end{cases}
\end{equation}

\noindent where $\sigma:\mathbb{H}_r\rightarrow\mathbb{C}\setminus[-1,1]$ is as defined as in the last paragraph of Section 2.1, such that $\sigma(z)=\exp(z)$ for $\textrm{Re}(z)>2\pi$. We have also emphasized the dependence in the definition of $g$ on several sets of parameters: $\lambda\in\mathbb{R}^+$, and $w, \delta, m$ as discussed in Section \ref{interpolation}. We have noted in (\ref{quasiregular_definition}), furthermore, that $w, \delta, m$ are allowed to depend on $n$. The points $z_n$ depend on $\lambda$ (see Subsection 2.2). We will use the notation $\bf{w}$ to denote the vector $(w_1, w_2, ...)$, and similarly for $\boldsymbol{\delta}$, $\boldsymbol{m}$. We will use either the notation $w_n$ or $\boldsymbol{w}(n)$ to denote the $n^{\textrm{th}}$ element of the vector $\boldsymbol{w}$, and similarly for $\boldsymbol{\delta}$, $\boldsymbol{m}$.

\begin{thm}\label{g_extension} There exist $m_0\in\mathbb{N}$, $\delta_0>0$, and  $k_0<1$ such that if $\boldsymbol{m}(n)>m_0$, $0\leq\boldsymbol{\delta}(n)<\delta_0$, and $\boldsymbol{w}(n)\in D(0,3/4)$ for all $n\in\mathbb{N}$, then, for any $\lambda>1$, $g$ as in (\ref{quasiregular_definition}) may be extended to a quasiregular map $g:\mathbb{C}\rightarrow\mathbb{C}$ such that $||g_{\overline{z}}/g_z||_{L^\infty(\mathbb{C})}<k_0$. The function $g:\mathbb{C}\rightarrow\mathbb{C}$ satisfies $g(-z)=g(z)$, $g(\overline{z})=\overline{g(z)}$ for all $z\in\mathbb{C}$, and the singular set of $g$ consists only of the critical values 

\[ \pm1 \emph{  and  } \left( w_n+\delta_n\left(\frac{-\delta_n}{m_n}\right)^{\left(\frac{1}{m_n-1}\right)}\left(\frac{m_n-1}{m_n}\right) \right)_{n=1}^{\infty}. \]

\end{thm}

\begin{proof} The bound $||g_{\overline{z}}/g_z||_{L^\infty(\cup D_n)}<k_0$ follows from considering $m_0$ as in Lemma \ref{interpolation_map}, and $\delta_0$, $k_0$ as in Proposition \ref{small_dilatation}. The extension of $g$ and the bound on $||g_{\overline{z}}/g_z||_{L^\infty(\mathbb{C}\setminus(\cup_nD_n))}$ are consequences of Theorem 7.2 of \cite{Bis15} as described in Section 17 of \cite{Bis15} (see also Section 3 of \cite{FGJ15}). The symmetry $g(-z)=g(z)$, $g(\overline{z})=\overline{g(z)}$ is built into the definition of $g$. The singular values 

\begin{equation}\label{expression_for_critical_values} \left( w_n+\delta_n\left(\frac{-\delta_n}{m_n}\right)^{\left(\frac{1}{m_n-1}\right)}\left(\frac{m_n-1}{m_n}\right) \right)_{n=1}^{\infty}. \end{equation}

\noindent arise from the critical values of $(g|_{D_n})_{n=1}^{\infty}$. That there are no other singular values of $g$ apart from $\pm1$ follows from Theorem 7.2 of \cite{Bis15}.
\end{proof}

\begin{rem}\label{changing_delta_w} Let $V:= \cup_{n=1}^{\infty} D_n$, and let $U:=\mathbb{C}\setminus \left( V \cup \textrm{conj}(V) \cup -V \cup -\textrm{conj}(V) \right)$, where $\textrm{conj}$ denotes complex conjugation.  Note that the extension of $g$ in $U$ is independent of a choice of $\boldsymbol{\delta}$, $\boldsymbol{w}$ since varying these parameters does not change the definition of $g$ on $\partial D_n$.
\end{rem}

\begin{definition} Let $\delta_0$ be as given in Theorem \ref{g_extension}. We call the parameters $\boldsymbol{\delta}$, $\boldsymbol{w}$ \emph{permissible} if $0\leq\boldsymbol{\delta}(n)<\delta_0$, and $\boldsymbol{w}(n)\in D(0,3/4)$ for all $n\in\mathbb{N}$.
\end{definition}

\begin{prop}\label{normalization} There exist $\lambda_0\in\mathbb{R}$, $\boldsymbol{m}_0\in\mathbb{N}^{\mathbb{N}}$ such that if $\boldsymbol{\delta}$, $\boldsymbol{w}$ are permissible, $\lambda>\lambda_0$, and $\boldsymbol{m}\geq\boldsymbol{m}_0$, then there exist constants $a_1, a_0, a_{-1} \in \mathbb{C}$ such that

\begin{equation}\label{normalization_equation} \phi(z)=a_1z+a_0+\frac{a_{-1}}{z} + O\left(\frac{1}{|z|^2}\right) \emph{   as   } z\rightarrow\infty,
\end{equation}

\noindent where $\phi$ is any quasiconformal mapping as in Theorem \ref{MRT} such that $g\circ\phi^{-1}$ is holomorphic.

\end{prop}

\begin{proof} This is a consequence of Theorem 7.2 of \cite{Bis15} and a Theorem of Dyn'kin \cite{MR1466801}. The statement (\ref{normalization_equation}) follows from \cite{MR1466801} once we know that 

\begin{equation}\label{normalization_equation_2} \int_{\textrm{supp}(g_{\overline{z}}) \cap (|z|>n)} \frac{\textrm{d}A(z)}{|z|^2} = O(e^{-cn}) \textrm{ as } n\rightarrow\infty
\end{equation}

\noindent for some $c>0$, where $\textrm{d}A$ is area measure on $\mathbb{C}$ (see also the discussion in Section 8 of \cite{Bis15}). The estimate (\ref{normalization_equation_2}) follows on $\textrm{supp}(g_{\overline{z}})\cap (|z|>n)\cap (\cup D_n)$ by requiring $\boldsymbol{m}_0(n) \rightarrow \infty$ as $n\rightarrow\infty$ sufficiently quickly by Proposition \ref{small_dilatation}. On  $\textrm{supp}(g_{\overline{z}})\cap(|z|>n)\setminus (\cup D_n)$,  (\ref{normalization_equation_2}) follows from the extension of $g$ as in Section 17 of \cite{Bis15} for sufficiently large $\lambda$, $\boldsymbol{m}$. We note furthermore that the constants in the $O(1/|z|^2)$ term are uniform over all such choices of $\lambda$, $\boldsymbol{m}$ and permissible $\boldsymbol{\delta}$, $\boldsymbol{w}$ by \cite{MR1466801}.

\end{proof}

\begin{rem} Given $\phi$ as in Proposition \ref{normalization} satisfying (\ref{normalization_equation}), we may normalize $\phi$ so that:

\begin{equation}\label{hydrodynamical}  \phi(z)=z+\frac{a}{z} + O\left(\frac{1}{|z|^2}\right) \emph{   as   } z\rightarrow\infty
\end{equation}

\noindent for some $a\in\mathbb{C}$. This is the normalization we will always use henceforth.

\end{rem}

\begin{prop}\label{close_to_id} For any $C>0$, $\varepsilon>0$, $R\geq1$, there exist $\lambda_0\in\mathbb{R}$, $\boldsymbol{m}_0\in\mathbb{N}^{\mathbb{N}}$, such that if $\lambda>\lambda_0$, $\boldsymbol{m}\geq\boldsymbol{m}_0$ and the parameters $\boldsymbol{\delta}$, $\boldsymbol{w}$ are permissible, then there exists a quasiconformal mapping $\phi: \mathbb{C} \rightarrow\mathbb{C}$ satisfying (\ref{hydrodynamical}) such that $g\circ\phi^{-1}$ is holomorphic and: 

\begin{equation}\label{qc_map_estimate} \left| \phi(z)-z \right| < \frac{C}{|z|} \emph{ for } |z|>R, \emph{ and }
\end{equation}
\begin{equation}\label{qc_map_estimate2} \left| \phi(z)-z \right| < \varepsilon \emph{ for all } z\in\mathbb{C}.
\end{equation}

\end{prop}

\begin{proof} This is a consequence of the normality of a family of $K$-quasiconformal mappings with the fixed normalization (\ref{hydrodynamical}) (see Section II.5 of \cite{MR0344463}). Indeed, suppose by way of contradiction that there exist $C>0$, $R>1$ such that there exist a sequence of arbitrarily large choices of  $\lambda\in\mathbb{R}$, $\boldsymbol{m}\in\mathbb{N}^{\mathbb{N}}$ so that (\ref{qc_map_estimate}) fails, where we can assume that each $\lambda>\lambda_0$, $\boldsymbol{m}>\boldsymbol{m_0}$ for $\lambda_0$, $\boldsymbol{m_0}$ as in Proposition \ref{normalization}. Label the associated quasiconformal mappings satisfying (\ref{hydrodynamical}) as $\phi_n$ (we are using Proposition \ref{normalization} to justify that each $\phi_n$ may be normalized to satisfy (\ref{hydrodynamical})). By normality there is a quasiconformal mapping $\psi:\mathbb{C}\rightarrow\mathbb{C}$ and a subsequence $\phi_{n_j}\rightarrow\psi$ uniformly on compact subsets of $\hat{\mathbb{C}}$ (in the spherical metric). Since $(\phi_{n_j})_{\overline{z}} \rightarrow 0$ a.e. as $j\rightarrow\infty$ (this follows from the extension of $g$ as in \cite{Bis15}), it follows from Theorem \ref{Lehto-Virtanen} that $\psi_{\overline{z}}/\psi_{z}=0$ a.e., and hence $\psi\equiv\textrm{identity}$. Thus $z(\phi_{n_j}(z)-z) \rightarrow 0$ uniformly on $\hat{\mathbb{C}}$ (in the spherical metric), contradicting the failure of (\ref{qc_map_estimate}) for all the maps $\phi_{n_j}$. A similar normal family argument ensures (\ref{qc_map_estimate2}).


\end{proof}

\begin{definition}\label{s_definition} Let $\lambda_0$, $\boldsymbol{m}_0$ be as given in Proposition \ref{close_to_id} for $\varepsilon=\varepsilon_0:=1/32$, and $C=R=1$. We call the parameters $\lambda$, $\boldsymbol{m}$ \emph{permissible} if $\lambda>\lambda_0$ and $\boldsymbol{m}(n)>\boldsymbol{m_0}(n)$, for all $n\in\mathbb{N}$.
\end{definition}

\begin{rem} Permissible parameters $\lambda$, $\boldsymbol{\delta}$, $\boldsymbol{m}$, $\boldsymbol{w}$ determine a quasiregular function $g$ via (\ref{quasiregular_definition}) and Theorem \ref{g_extension} such that $\phi^{-1}$ satisfies (\ref{qc_map_estimate}) and (\ref{qc_map_estimate2}) with $C=R=1$ and $\varepsilon=1/32$, where $\phi^{-1}$ is a quasiconformal map normalized as in (\ref{hydrodynamical}) such that $g\circ\phi^{-1}$ is holomorphic.

\end{rem}

To summarize our discussions up until this point, we now have a family of quasiregular functions $g:\mathbb{C}\rightarrow\mathbb{C}$ given as in (\ref{quasiregular_definition}) and Theorem \ref{g_extension} depending on parameters  $\lambda$, $\boldsymbol{\delta}$, $\boldsymbol{m}$, $\boldsymbol{w}$. If, moreover, these parameters are permissible, we then have estimates on the straightening map of $g$ as in Theorem \ref{MRT}. All subsequent considerations will be to the purpose of showing that there is some choice of these parameters for which the associated entire function $f:=g\circ\phi^{-1}$ is as in Theorem \ref{mainthm}.

\begin{definition}\label{mu_definition} Define a sequence of real numbers $(\mu_n)_{n=1}^{\infty} \rightarrow0$ as $n\rightarrow\infty$ such that any quasiconformal homeomorphism $\phi$ satisfying (\ref{qc_map_estimate}) and (\ref{qc_map_estimate2}) satisfies:

\begin{align} \phi( D(z_n, 1-2\mu_n)  ) \subset D(z_n, 1-\mu_n)\textrm{ and }  \phi(  \mathbb{C}\setminus D(z_n, 1+2\mu_n)) \subset   \mathbb{C}\setminus D(z_n, 1+\mu_n). 
\end{align}

\noindent We will also use the notation 

\begin{align} D^{+}_{n}:=D(z_n, 1+\mu_n)\textrm{,    }   D^{++}_{n}:=D(z_n, 1+2\mu_n),  \nonumber
\\ D^{-}_{n}:=D(z_n, 1-\mu_n)\textrm{,     } D^{--}_{n}:=D(z_n, 1-2\mu_n). \nonumber \end{align}

\end{definition}

\begin{rem} We will henceforth begin considering the local inverse map $g^{-1}$, which we will always define in a neighborhood of $g(x)=y$ with $x,y>0$ and so that $g^{-1}(y)=x$. There are no positive critical points of $g$ by (\ref{quasiregular_definition}) so that this inverse is always well defined, at least locally near $y$. The same remarks apply to the local inverse $f^{-1}$.

\end{rem}

\begin{prop}\label{prop_on_derivative_of_f} Let $n\geq1$, and suppose that $g'(x)\geq2$ for $x\geq1/32$ and that $\lambda$, $\boldsymbol{\delta}$, $\boldsymbol{m}$, $\boldsymbol{w}$ are permissible. Assume furthermore that $\textrm{supp}(g_{\overline{z}})\cap S^+\subset\{z\in S^+: \dist(z,\partial S^+) < 1/16\}$. Then 

\begin{align}\label{derivative_of_f} {\scriptstyle \left( \frac{(1/8)^2(1/4-2\varepsilon_0)}{(3/8)^2(1/4)} \right)^n \cdot \prod_{k=1}^n (g^{-1})'(g^k(1/2)+2\varepsilon_0) \leq \left| (f^{-n})'(g^n(1/2)) \right| 
 \leq \left( \frac{(5/8)^2(1/4+2\varepsilon_0)}{(3/8)^2(1/4)\lambda} \right)^n\cdot\frac{1}{\lambda-(\varepsilon_0/\lambda^{n-2})} }.
\end{align}


\end{prop}

\begin{proof} We first bound $\phi'(x)$ for $x\geq1/2-\varepsilon_0$ by using (\ref{qc_map_estimate2}) and Theorem \ref{thm:Koebe}(b) with $F=\phi$, $a=x$, $z\in\partial D(x,1/4)$, $r=3/8$:

\begin{equation}\label{bound_on_derivative_of_phi} \frac{(1/8)^2(1/4-2\varepsilon_0)}{(3/8)^2(1/4)} \leq \left| \phi'(x) \right| \leq \frac{(5/8)^2(1/4+2\varepsilon_0)}{(3/8)^2(1/4)}.
\end{equation}

\noindent Since $g'(x)\geq2$ for $x\geq1/32$, the Mean Value Theorem implies that $[x-2\varepsilon_0, x+2\varepsilon_0] \subset g([g^{-1}(x)-\varepsilon_0, g^{-1}(x)+\varepsilon_0])$.  Hence, using the fact that $g'$ is increasing (and hence $(g^{-1})'$ decreasing) for $x\geq0$, we can calculate:


\begin{equation}\label{bound_on_derivative_of_f} \left| (f^{-n})'(g^n(1/2)) \right| \geq \left( \frac{(1/8)^2(1/4-2\varepsilon_0)}{(3/8)^2(1/4)} \right)^n \cdot \prod_{k=1}^n (g^{-1})'(g^k(1/2)+2\varepsilon_0).
\end{equation}


\noindent The upper bound

\begin{equation}\label{upper_bound_on_f_inverse} \left| (f^{-n})'(g^n(1/2)) \right| 
\\ \leq \left(\frac{(5/8)^2(1/4+2\varepsilon_0)}{(3/8)^2(1/4)}\right)^n  \cdot\prod_{k=1}^n (g^{-1})'(g^k(1/2)-2\varepsilon_0) 
\end{equation}

\noindent is calculated similarly. We abbreviate $C_0:=((5/8)^2(1/4+2\varepsilon_0))/((3/8)^2(1/4))$, and calculate

\begin{align} (C_0)^n\cdot\prod_{k=1}^n (g^{-1})'(g^k(1/2)-2\varepsilon_0) = \frac{(C_0)^n}{\prod_{k=1}^n g'(g^{-1}(g^k(1/2)-2\varepsilon_0))} \leq \frac{(C_0)^n}{ g'(g^{-1}(g^n(1/2)-2\varepsilon_0))} \leq \nonumber
\\ \frac{(C_0)^n}{ g'(g^{n-1}(1/2)-\varepsilon_0)} \leq \frac{(C_0)^n}{ \lambda^2(g^{n-1}(1/2)-\varepsilon_0)} \leq \frac{(C_0)^n}{ \lambda^{n+1}(1-(\varepsilon_0/\lambda^{n-1}))}, \nonumber 
\end{align}

\noindent where we have used the formula $g'(x)=\lambda\sinh(\lambda\sinh(x))\exp(x)$, and the facts that $\sinh(x)\geq x$ and $\exp(x)\geq1$ for $x\geq0$. The last inequality follows because $g(x)>\lambda x$ for $x\geq 1/2$, which is a consequence of the Fundamental Theorem of Calculus and the fact that $g(1/2)>\lambda(1/2)$ and $g'(x)\geq\lambda^2x>\lambda$ for $x\geq1/2$ and $\lambda$ sufficiently large.


\end{proof}

\begin{rem}\label{choice_of_lambda} We will henceforth fix $\lambda=\lambda_0$ as in Proposition \ref{close_to_id}, with several extra  conditions: we assume that $\lambda_0$ is sufficiently large so that $g'(x)=\lambda\sinh(\lambda\sinh(x))\exp(x)\geq2$ for $x\geq 1/32$, and that $\lambda_0$ is sufficiently large so that $g^n(1/2)\rightarrow\infty$ as $n\rightarrow\infty$ (see Lemma 3.2 of \cite{FGJ15}). Furthermore, we assume $\lambda$ is sufficiently large so that $\textrm{supp}(g_{\overline{z}})\cap S^+\subset\{z\in S^+: \dist(z,\partial S^+) < 1/16\}$ (see the definition of $T(r_0)$ as in Theorem \ref{folding}). Lastly, we assume that $\lambda=\lambda_0$ is sufficiently large so that the right-hand side of (\ref{derivative_of_f}) tends to $0$ as $n\rightarrow\infty$. Note that the upper and lower bounds in (\ref{derivative_of_f}) are independent of permissible $\boldsymbol{\delta}$, $\boldsymbol{m}$, $\boldsymbol{w}$. Furthermore, observe that the map $g|_{S+}$ and the points $z_n$ as in (\ref{quasiregular_definition}) are now both fixed henceforth as they depend only on $\lambda$.

\end{rem}

\begin{definition} Define the sequence $(p_n)_{n=1}^{\infty}$ such that $|z_{p_n}-g^n(1/2)|$ is minimized. \end{definition}

\begin{cor}\label{dist_to_1/2} There exists $n'\in\mathbb{N}$ such that if $\boldsymbol{\delta}$, $\boldsymbol{m}$, and $\boldsymbol{w}$ are permissible, then $f^{-n}(D_{p_n})\subset D(1/2, 1/8) \subset D(0, 3/4)$ for all $n>n'$.
\end{cor}

\begin{proof} First note that $|f^{-n}(g^n(1/2))-1/2| < 2\varepsilon_0=1/16$ by the proof of Proposition \ref{prop_on_derivative_of_f}. We use (\ref{derivative_of_f}) and Theorem \ref{thm:Koebe}(b) with $F=f^{-n}$, $z\in D_{p_n}$, $r=10$, $a=g^n(1/2)$ to yield:

\begin{equation}\label{dist_bound} \left| f^{-n}(z) - f^{-n}(g^n(1/2)) \right| \leq \frac{10^2\left( 3\pi/2 + 1 \right)}{\left(10-\left( 3\pi/2 + 1 \right)\right)^2}  \left( \frac{(5/8)^2(1/4+2\varepsilon_0)}{(3/8)^2(1/4)\lambda} \right)^n\cdot\frac{1}{\lambda-(\varepsilon_0/\lambda^{n-2})} . \nonumber
\end{equation}

\noindent Note that the right-hand side tends to $0$ as $n\rightarrow\infty$ by the choice of $\lambda=\lambda_0$ as in Remark \ref{choice_of_lambda}, so that for large $n$ we have the conclusion of Corollary \ref{dist_to_1/2} as needed.

\end{proof}

\section{Performing Surgery to Yield a Wandering Domain}\label{final_section}

The goal of the present Section is to show that there exists some particular choice of the parameters $\boldsymbol{\delta}$, $\boldsymbol{m}$, $\boldsymbol{w}$ such that the associated entire function has a wandering domain (in the next Section we will prove the univalence statement in Theorem \ref{mainthm}). It will be necessary to first fix $\boldsymbol{m}$ in Proposition \ref{selection_of_r_m} below, before choosing $\boldsymbol{\delta}$, $\boldsymbol{w}$ in Proposition \ref{final_prop}. The purpose of the technical inequalities (\ref{eq1}) and (\ref{eq2}) of Proposition \ref{selection_of_r_m} is illustrated in Figure \ref{fig:annuli}.

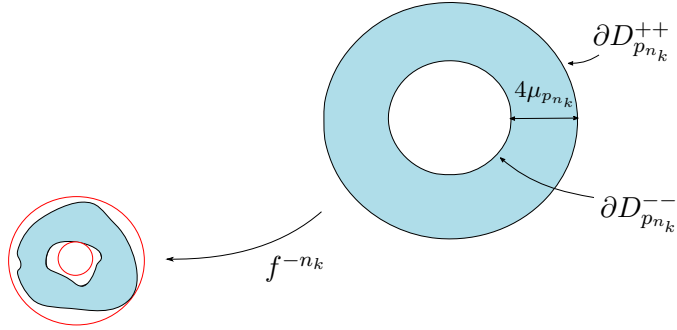
\begin{figure}[!htb]
\centering
\includegraphics[width=0.6\textwidth]{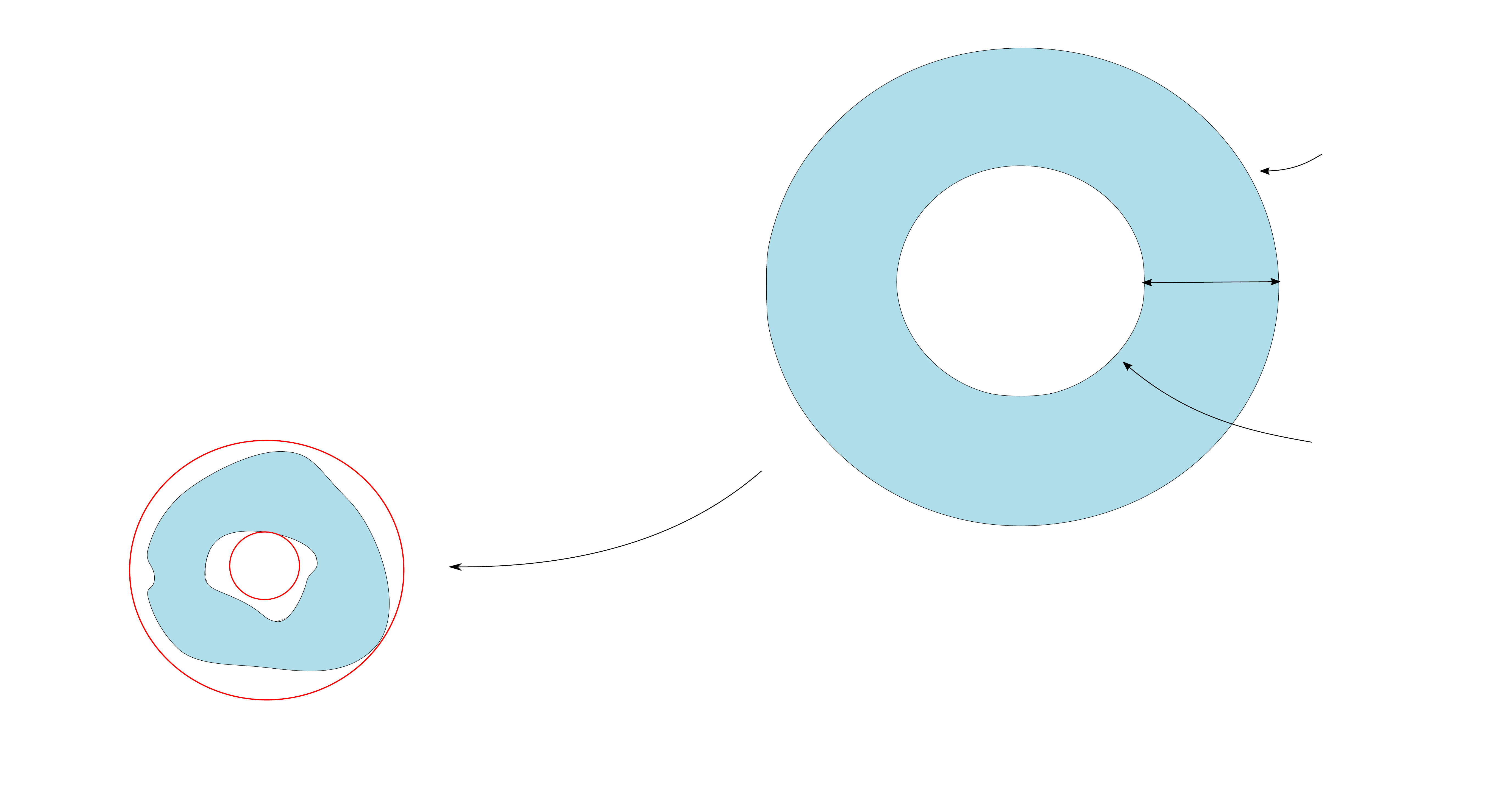}
\setlength{\unitlength}{0.6\textwidth}
\put(-0.13,0.45){\small $\partial D_{p_{n_k}}^{++}$}
\put(-0.12,0.22){\small $\partial D_{p_{n_k}}^{--}$}
\put(-0.235,0.38){\scriptsize $4 \mu_{p_{n_k}}$}
\put(-0.57,0.14){\small $f^{-n_k}$}
\caption{ \small The inner, outer Euclidean circles surrounding $f^{-n_k}(\partial D_{p_{n_k}})$ have radii corresponding to the min, max terms appearing in inequalities (\ref{eq1}), (\ref{eq2}) of Proposition \ref{selection_of_r_m}, respectively. The negative of the other term appearing on the left-hand side of inequality (\ref{eq1}) is an upper bound for the modulus of $(z-z_{p_{n_{k-1}}})^{m_{p_{n_{k-1}}}}+\delta (z-z_{p_{n_{k-1}}})$ for $z\in D_{p_{n_{k-1}}}^{-}$ and a choice of $\delta=\boldsymbol{\delta}(p_{n_{k-1}})$ that will be made in Proposition \ref{final_prop}. Thus inequality (\ref{eq1}) will be used to guarantee that $f(D_{p_{n_{k-1}}}^{-}) \subset f^{-n_k}(D_{p_{n_k}}^{--})$. Similarly, (\ref{eq2}) will be used to guarantee that the critical values of $f$ associated to $D_{p_{n_{k-1}}}$ land outside $f^{-n_k}(D_{p_{n_k}}^{++})$.  }

\label{fig:annuli}
\end{figure}

\begin{prop}\label{selection_of_r_m} There exist a subsequence $(n_k)_{k=1}^{\infty}$ of natural numbers, a sequence $(C_{n_k})_{k=1}^{\infty}$ of positive real numbers, and a permissible parameter $\boldsymbol{m}\in\mathbb{N}^{\mathbb{N}}$ such that: for any choice of permissible $\boldsymbol{w}$, $\boldsymbol{\delta}$ one has:


\begin{align}  
\min_{|z-z_{p_{n_k}}|=1-2\mu_{p_{n_k}}} \left| f^{-n_k}(z)- f^{-n_k}(z_{p_{n_k}})\right| - \left(1-\mu_{{p_{n_{k-1}}}}\right)^{m_{p_{n_{k-1}}}} \label{eq1}
 -  \phantom{Asdfasdfasds}
\\ \left( \frac{2-\mu_{{p_{n_{k-1}}}}}{2-2\mu_{{p_{n_{k-1}}}}} \right) \cdot\left|\left( f^{-n_k} \right)'(g^{n_k}(1/2))\right|\cdot\left(1-\mu_{p_{n_{k-1}}}\right) > C_{n_{k}} > 0\textrm{, }\forall k\geq2,  \phantom{Asdfasdf} \nonumber
\\ \left( \frac{2-\mu_{{p_{n_{k-1}}}}}{2-2\mu_{{p_{n_{k-1}}}}} \right)\cdot \left|\left( f^{-n_k} \right)'(g^{n_k}(1/2))\right| \cdot \left( \left( \frac{2-\mu_{{p_{n_{k-1}}}}}{2-2\mu_{{p_{n_{k-1}}}}} \right)/m_{p_{n_{k-1}}} \right)^{1/(m_{p_{n_{k-1}}}-1)}\cdot\left(\frac{m_{p_{n_{k-1}}}-1}{m_{p_{n_{k-1}}}}\right)  \label{eq2}
\\ - \max_{|z-z_{n_k}|=1+2\mu_{n_k}} \left| f^{-n_k}(z)- f^{-n_k}(z_{p_{n_k}})\right| > C_{n_k} > 0 \textrm{, }\forall k\geq2. \phantom{Asdadsfasfasdf} \nonumber
\end{align}

\noindent Furthermore, for fixed $k\geq2$, any $k_0$-quasiconformal mapping $\psi$ with dilatation supported in $1\geq|z-z_{p_{n_{k-1}}}| > 1-4\delta_0/m_{p_{n_{k-1}}}$ and normalization (\ref{hydrodynamical}) satisfies : 

\begin{align} \left| \psi(z)-z \right| < \min\left(\frac{C_{n_{l}}}{2^{k}}\right)_{l=2}^{k} \textrm{ for all } z\in\mathbb{C}. \label{eq3}
\end{align}


\end{prop}

\begin{proof} 

Define $n_1:=1$. By Theorem \ref{thm:Koebe}(b), (c), one has that for $z\in\partial D_{p_n}^{--}$,

\begin{align}\label{meq1}  \left| f^{-n}(z)- f^{-n}(z_{p_n})\right| \geq  \left|(f^{-n})'(g^n(1/2))\right| \left|\frac{(f^{-n})'(z_{p_n})}{(f^{-n})'(g^n(1/2))}\right|  \frac{\dist_n^2\cdot(1-2\mu_{p_n})}{\left( \dist_n+(1-2\mu_{p_n}) \right)^2}
\\ \label{meq1*} \geq \left|(f^{-n})'(g^n(1/2))\right|\frac{1-\left| 2\pi/\dist_n \right|}{\left(1+\left| 2\pi/\dist_n \right|\right)^3}\frac{\dist_n^2\cdot(1-2\mu_{p_n})}{\left( \dist_n+(1-2\mu_{p_n}) \right)^2},
\end{align}

\noindent where $\dist_n$ is the radius of conformality for the inverse branch $f^{-n}$ defined in a neighborhood of $z_{p_n}$. Similarly, for $z\in\partial D_{p_n}^{++}$,

\begin{align}\label{meq2} \left| f^{-n}(z)- f^{-n}(z_{p_n})\right| \leq  \left|(f^{-n})'(g^n(1/2))\right| \left|\frac{(f^{-n})'(z_{p_n})}{(f^{-n})'(g^n(1/2))}\right|  \frac{\dist_n^2\cdot(1+2\mu_{p_n})}{\left( \dist_n-(1+2\mu_{p_n}) \right)^2}
\\ \label{meq2*} \leq \left|(f^{-n})'(g^n(1/2))\right| \frac{1+\left| 2\pi/\dist_n \right|}{\left(1-\left| 2\pi/\dist_n \right|\right)^3}  \frac{\dist_n^2\cdot(1+2\mu_{p_n})}{\left( \dist_n-(1+2\mu_{p_n}) \right)^2}.
\end{align} 

\noindent The second and third terms in the products in both (\ref{meq1*}) and (\ref{meq2*}) tend to $1$ as $n\rightarrow\infty$, since $\dist_n\rightarrow\infty$ and $\mu_{p_n}\rightarrow0$ as $n\rightarrow\infty$. Thus we may choose $n_2$ such that 

\begin{align}\label{independent_of_f1} \frac{1-\left| 2\pi/\dist_n \right|}{\left(1+\left| 2\pi/\dist_n \right|\right)^3}  \frac{\dist_n^2\cdot(1-2\mu_{p_n})}{\left( \dist_n+(1-2\mu_{p_n}) \right)^2} \geq \frac{2-\mu_{p_{1}}}{2}\textrm{, and }
\\ \label{independent_of_f2}  \frac{1+\left| 2\pi/\dist_n \right|}{\left(1-\left| 2\pi/\dist_n \right|\right)^3}  \frac{\dist_n^2\cdot(1+2\mu_{p_n})}{\left( \dist_n-(1+2\mu_{p_n}) \right)^2} \leq  \frac{2-\mu_{p_{1}}}{2-2\mu_{p_{1}}},
\end{align}

\noindent with $n=n_2$. We have now chosen $n_2$, and proceed to choose $C_{n_2}$, and then $m_{p_1}$. Using (\ref{independent_of_f1}), we note that

\begin{align}\label{something1} {\scriptstyle \left|(f^{-n_2})'(g^{n_2}(1/2))\right|\bigg(  \frac{1-\left| 2\pi/\dist_{n_2} \right|}{\left(1+\left| 2\pi/\dist_{n_2} \right|\right)^3}   \frac{\dist_{n_2}^2\cdot(1-2\mu_{p_{n_2}})}{\left( \dist_{n_2}+(1-2\mu_{p_{n_2}}) \right)^2} - \left( \frac{2-\mu_{{p_{{1}}}}}{2} \right)\bigg) - \left(1-\mu_{{p_{1}}}\right)^{m_{{p_1}}} } 
\\ {\scriptstyle \geq c_{n_2}\bigg( \frac{1-\left| 2\pi/\dist_{n_2} \right|}{\left(1+\left| 2\pi/\dist_{n_2} \right|\right)^3}   \frac{\dist_{n_2}^2\cdot(1-2\mu_{p_{n_2}})}{\left( \dist_{n_2}+(1-2\mu_{p_{n_2}}) \right)^2} - \left( \frac{2-\mu_{{p_{{1}}}}}{2} \right)\bigg) - \left(1-\mu_{{p_{1}}}\right)^{m_{{p_1}}}} \xrightarrow{m_{p_1}\rightarrow\infty} \nonumber \phantom{as}
\\  {\scriptstyle  \nonumber c_{n_2}\bigg( \frac{1-\left| 2\pi/\dist_{n_2} \right|}{\left(1+\left| 2\pi/\dist_{n_2} \right|\right)^3}  \frac{\dist_{n_2}^2\cdot(1-2\mu_{p_{n_2}})}{\left( \dist_{n_2}+(1-2\mu_{p_{n_2}}) \right)^2} -  \left( \frac{2-\mu_{{p_{{1}}}}}{2} \right) \bigg) > 0, \phantom{asasdfsdasdf} }
\end{align}

\noindent where $c_{n_2}$ is a lower bound for $\left|(f^{-n_2})'(g^{n_2}(1/2))\right|$ as given by (\ref{derivative_of_f}). Moreover, by (\ref{independent_of_f2}), we see that:

\begin{align}\label{something2} {\scriptstyle \left|(f^{-n_2})'(g^{n_2}(1/2))\right|\bigg( \left( \frac{2-\mu_{{p_{1}}}}{2-2\mu_{{p_{1}}}} \right)\cdot \left( \left( \frac{2-\mu_{{p_{1}}}}{2-2\mu_{{p_{1}}}} \right)/m_{p_1} \right)^{1/(m_{p_1}-1)}\cdot\left(\frac{m_{p_1}-1}{m_{p_1}}\right) -  \frac{1+\left| 2\pi/\dist_{n_2} \right|}{\left(1-\left| 2\pi/\dist_{n_2} \right|\right)^3} \frac{\dist_{n_2}^2\cdot(1+2\mu_{p_{n_2}})}{\left( \dist_{n_2}-(1+2\mu_{p_{n_2}}) \right)^2} \bigg) }  
\\ {\scriptstyle \xrightarrow{ m_{p_1}\rightarrow\infty }   \left|(f^{-n_2})'(g^{n_2}(1/2))\right|\bigg(   \left( \frac{2-\mu_{{p_{1}}}}{2-2\mu_{{p_{1}}}} \right) - \frac{1+\left| 2\pi/\dist_{n_2} \right|}{\left(1-\left| 2\pi/\dist_{n_2} \right|\right)^3}  \frac{\dist_{n_2}^2\cdot(1+2\mu_{p_{n_2}})}{\left( \dist_{n_2}-(1+2\mu_{p_{n_2}}) \right)^2} \bigg) }  \nonumber \phantom{asdfsasd}
\\ {\scriptstyle \geq c_{n_2} \bigg(   \left( \frac{2-\mu_{{p_{1}}}}{2-2\mu_{{p_{1}}}} \right) -  \frac{1+\left| 2\pi/\dist_{n_2} \right|}{\left(1-\left| 2\pi/\dist_{n_2} \right|\right)^3} \frac{\dist_{n_2}^2\cdot(1+2\mu_{p_{n_2}})}{\left( \dist_{n_2}-(1+2\mu_{p_{n_2}}) \right)^2} \bigg) > 0. } \phantom{asasdssdsdfsasd} \nonumber
\end{align}

\noindent Thus we define $2C_{n_2}$ to be the minimum of the bottom expressions in (\ref{something1}) and (\ref{something2}),  so that (\ref{eq1}) and (\ref{eq2}) holds for $k=2$ and all sufficiently large $m_{p_1}$. The inequality (\ref{eq3}) holds for sufficiently large $m_{p_1}$ by Proposition \ref{small_dilatation} and a normal family argument similar to the proof of Proposition \ref{normalization}. 

To summarize, we first fixed $n_2$, and then fixed $C_{n_2}$ depending on our choice of $n_2$, and then fixed $m_{p_1}$ depending on our choices of $n_2$, $C_{n_2}$. We now proceed iteratively to define $n_3$ so that (\ref{independent_of_f1}) and (\ref{independent_of_f2}) are satisfied with $n=n_3$ and the right hand-side of (\ref{independent_of_f1}) replaced with $(2-\mu_{p_{n_2}})/2$, and the right-hand side of (\ref{independent_of_f2}) replaced with $(2-\mu_{p_{n_2}})/(2-2\mu_{p_{n_2}})$. A completely analogous argument to the one given in the above paragraphs guarantees $C_{n_3}$ and $m_{p_{n_2}}$ so that (\ref{eq1}), (\ref{eq2}), (\ref{eq3}) hold for $k=3$. This describes an inductive procedure which defines the sequences $(n_k)$, $(C_{n_k})$, $(m_{p_{n_k}})$ for which (\ref{eq1}), (\ref{eq2}), (\ref{eq3}) hold for all $k\geq2$ and any permissible $\boldsymbol{\delta}$, $\boldsymbol{w}$ provided that $(m_{p_{n_k}})$ extends to a permissible $\boldsymbol{m}$, which we can ensure by defining  $\boldsymbol{m}(l)=\boldsymbol{m}_0(l)$ for $l\in\mathbb{N}\setminus(p_{n_k})_{k=1}^{\infty}$.

\end{proof}

\begin{rem} We henceforth fix $\boldsymbol{m}$ as in Proposition \ref{selection_of_r_m}, and continue to use the sequences $(n_k)_{k=1}^{\infty}$, $(C_{n_k})_{k=1}^{\infty}$ as defined in Proposition \ref{selection_of_r_m}. 

\end{rem}

\begin{prop}\label{final_prop} There exist permissible $\boldsymbol{w}$, $\boldsymbol{\delta}$ such that $f:=g\circ\phi^{-1}$ has a wandering component containing $\phi(D_{p_1}^-)$. 
\end{prop}

\begin{proof}

Our strategy will be to start with the choice $\boldsymbol{w}=\boldsymbol{0}$, $\boldsymbol{\delta}=\boldsymbol{0}$, and  inductively redefine the sequences $\boldsymbol{w}(p_{n_k})$, $\boldsymbol{\delta}(p_{n_k})$ so that the the choice of parameters $\boldsymbol{w}(p_{n_k})$, $\boldsymbol{\delta}(p_{n_k})$, and $\boldsymbol{w}(l)=\boldsymbol{\delta}(l)=0$ for $l\in\mathbb{N}\setminus(p_{n_k})_{k=1}^\infty$ corresponds to a quasiregular function $g$ such that $f:=g\circ\phi^{-1}$ has a wandering domain. As already mentioned, the choices $\boldsymbol{w}=\boldsymbol{0}$, $\boldsymbol{\delta}=\boldsymbol{0}$ determine a quasiregular function $g_0$ via Theorem \ref{g_extension}, and thus a quasiconformal map $\phi_0$ via Theorem \ref{MRT} normalized as in (\ref{hydrodynamical}). Let $f_0:=g_0\circ\phi_0^{-1}$. Similarly, let $g_1$, $\phi_1$ be the functions associated to the choices 

\begin{equation}\label{first_choice} \boldsymbol{w}(p_1)=f_0^{-n_2}(z_{p_{n_2}})\textrm{,   } \boldsymbol{\delta}(p_1)=\frac{2-\mu_{p_1}}{2-2\mu_{p_1}}(f_0^{-n_2})'(g_0^{n_2}(1/2))\textrm{, and } \boldsymbol{w}(l)=\boldsymbol{\delta}(l)=0 \textrm{ for } l\not=p_1. \end{equation}

\begin{figure}[!htb]
\centering
\includegraphics[width=0.8\textwidth]{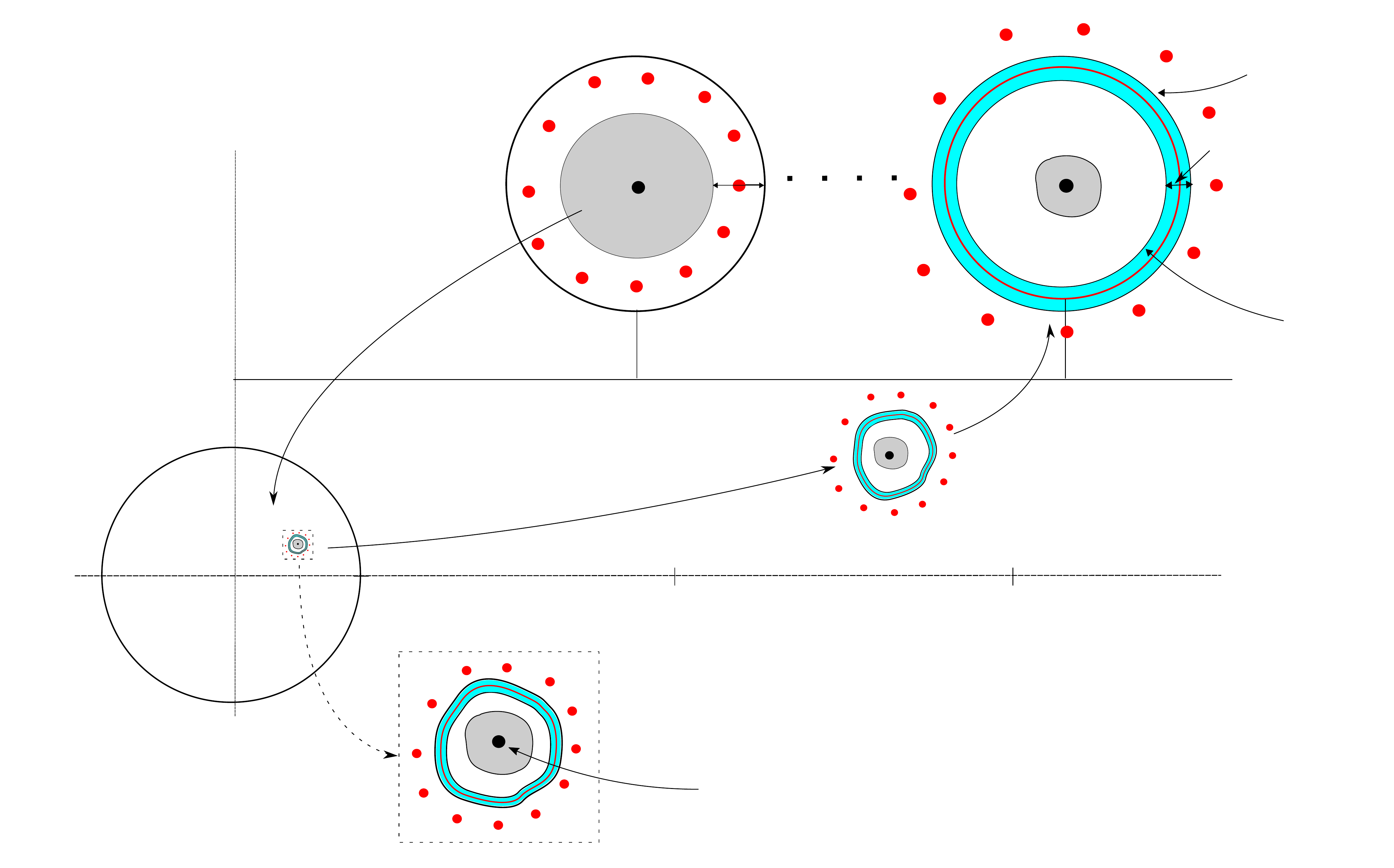}
\setlength{\unitlength}{0.8\textwidth}
\put(-0.12,0.52){\scriptsize $4\mu_{p_{n_2}}$}
\put(-0.485,0.51){\scriptsize $\mu_{p_1}$}
\put(-0.49,.05){\scriptsize $w_{p_1}:=f_0^{-n_2}(z_{p_{n_2}})$}
\put(-0.6,0.275){\footnotesize $f_1^{n_2-1}$}
\put(-0.8,0.4){\footnotesize $f_1\circ\phi_1$}
\put(-0.25,0.32){\footnotesize $f_1$}
\put(-0.56,0.51){\scriptsize $D_{p_1}^{-}$}
\put(-0.095,0.575){\scriptsize $\partial D_{p_{n_2}}^{++}$}
\put(-0.075,0.375){\scriptsize $\partial D_{p_{n_2}}^{--}$}
\caption{ \small This Figure illustrates the key relations (\ref{f_1_fits_inside}) and (\ref{critical_point_for_f_1}) for the function $f_1$, which in turn are consequences of the inequalities (\ref{eq1}) and (\ref{eq2}) of Proposition \ref{selection_of_r_m}.}
\label{fig:perfectfit}
\end{figure}
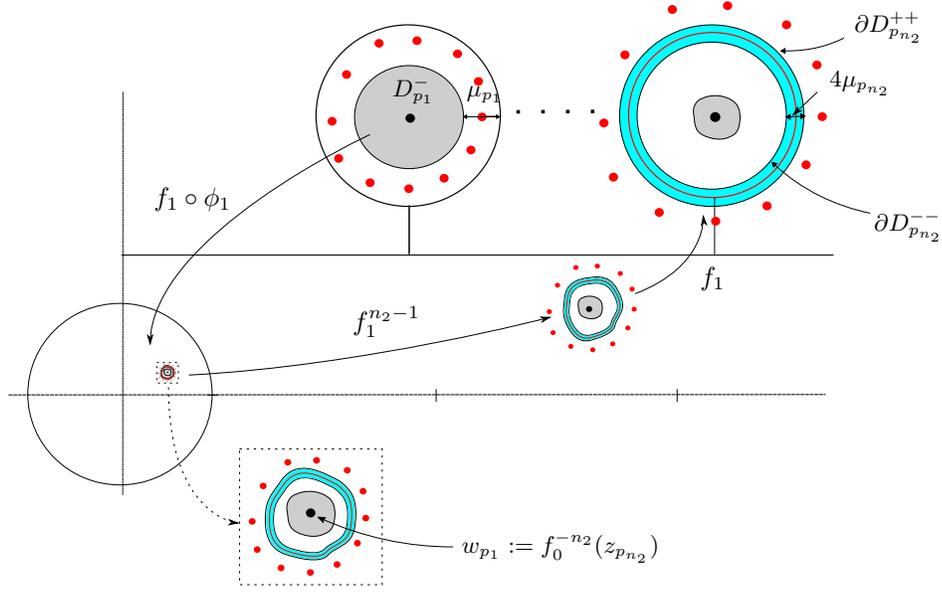

\noindent See Figure \ref{fig:perfectfit} for an illustration of the relations (\ref{f_1_fits_inside}) and (\ref{critical_point_for_f_1}) for $f_1$ we wish to establish below.  By (\ref{eq1}), $f_1(\phi_1( D_{p_{1}}^{-} )) \subset f_0^{-n_2}( D_{p_{n_2}}^{--} )$, and moreover $\dist( f_1(\phi_1( D_{p_{1}}^{-} )) ,  \partial f_0^{-n_2}( D_{p_{n_2}}^{--} ) ) > C_{n_2}$, where we are using the notation $\partial U$ to denote the boundary of a region $U$. By (\ref{eq3}), $|\phi_0\circ\phi_1^{-1}(z)-z|<C_{n_2}/4$, and thus $\dist( \phi_0\circ\phi_1^{-1} \circ f_1(\phi_1( D_{p_{1}}^{-} ) ) , \partial f_0^{-n_2}( D_{p_{n_2}}^{--} ) ) > 3C_{n_2}/4$. Thus, since $g_1(z)=g_0(z)$ for $z\not\in D_{p_1}$, we have that:

\[ f_1\left( f_1\left(\phi_1( D_{p_{1}}^{-} \right) \right) = g_1\circ\phi_0^{-1}\circ\phi_0\circ\phi_1^{-1}\left(f_1\left(\phi_1\left( D_{p_{1}}^{-} \right) \right)\right) \subset f_0^{-n_2+1}\left( D_{p_{n_2}}^{--} \right). \]

\noindent Moreover, $\dist( f_1( f_1(\phi_1( D_{p_{1}}^{-} ) ) , \partial f_0^{-n_2+1}( D_{p_{n_2}}^{--} ) ) > C_{n_2}$ by the expanding properties of $g_0$ in $S^+$. Iterating this argument and using (\ref{eq1}) and (\ref{eq3}), we see that:

\begin{equation}\label{f_1_fits_inside} f_1^{n_2}\left(\phi_1\left( D_{p_{1}}^{-} \right)  \right) \subset  D_{p_{n_2}}^{--}.  \end{equation}


Let $v$ be a critical value of $f_1$ associated to a critical point of $g_1|_{D_{p_{1}}}$.  By (\ref{expression_for_critical_values}), (\ref{eq2}), and (\ref{first_choice}), $v \not\in f_0^{-n_2}(D^{++}_{p_{n_2}})$, and moreover $\dist(v, \partial f_0^{-n_2}(D^{++}_{p_{n_2}}))>C_{n_2}$. Again, we note that $|\phi_0\circ\phi_1^{-1}(z)-z|<C_{n_2}/4$ by  (\ref{eq3}), so that $\dist(\phi_0\circ\phi_1^{-1}(v), \partial f_0^{-n_2}(D^{++}_{p_{n_2}}))>3C_{n_2}/4$, and hence:

\[ f_1(v) = g_1 \circ \phi_0^{-1}\circ\phi_0\circ\phi_1^{-1}(v) \not \in f_0^{-n_2+1}\left(D^{++}_{p_{n_2}}\right), \textrm{ and } \dist\left(  f_1(v), \partial f_0^{-n_2+1}\left(D^{++}_{p_{n_2}}\right)  \right) > C_{n_2}.   \] 

\noindent Iterating this argument, we see that 

\begin{equation}\label{critical_point_for_f_1} f_1^{n_2}(v)\not\in D^{++}_{p_{n_2}}\end{equation}

\noindent  for any critical value $v$ associated to a critical point of $g_1|_{D_{p_1}}$.

We proceed to define the quasiregular function $g_2$ as associated to the parameters $\boldsymbol{w}(p_1)$, $\boldsymbol{\delta}(p_1)$ as defined in (\ref{first_choice}), and:

\[ \boldsymbol{w}(p_{n_2})=f_1^{-n_3}(z_{p_{n_3}})\textrm{,   } \boldsymbol{\delta}(p_{n_2})=\frac{2-\mu_{p_{n_2}}}{2-2\mu_{p_{n_2}}}(f_1^{-n_3})'(g_1^{n_3}(1/2))\textrm{, and } \boldsymbol{w}(l)=\boldsymbol{\delta}(l)=0 \textrm{ for } l\not\in\{p_1, p_{n_2}\}. \] 

\noindent The functions $\phi_2$ and $f_2:=g_2\circ\phi_2^{-1}$ are defined analogously. The arguments that 

\begin{equation}\label{second_choice} f_2^{n_3}\left(\phi_2\left( D_{p_{n_2}}^{-} \right)  \right) \subset  D_{p_{n_3}}^{--}\textrm{ and } f_2^{n_3}(v)\not\in D^{++}_{p_{n_3}} \end{equation}

\noindent for any critical value $v$ of $f_2$ associated to $g_2|_{D_{p_{n_2}}}$ are identical to the arguments given for (\ref{f_1_fits_inside}) and (\ref{critical_point_for_f_1}).  We will verify that the relations (\ref{f_1_fits_inside}) and (\ref{critical_point_for_f_1}) still hold with $f_1$, $\phi_1$ replaced by $f_2$, $\phi_2$. Well $f_2(\phi_2( D_{p_{1}}^{-} )) = g_1(D_{p_{1}}^{-} )$, so that we have already verified $f_2(\phi_2( D_{p_{1}}^{-} )) \subset  f_0^{-n_2}( D_{p_{n_2}}^{--} )$ and $\dist( f_2(\phi_2( D_{p_{1}}^{-} )) ,  \partial f_0^{-n_2}( D_{p_{n_2}}^{--} ) ) > C_{n_2}$. By (\ref{eq1}), we know that $|\phi_1 \circ \phi_2^{-1}(z)-z|<C_{n_2}/8$, and thus $\dist( \phi_0\circ\phi_1^{-1} \circ \phi_1 \circ \phi_2^{-1} \circ f_2(\phi_2( D_{p_{1}}^{-} ) ) , \partial f_0^{-n_2}( D_{p_{n_2}}^{--} ) ) > 5C_{n_2}/8$. Hence:

\begin{equation} f_2 \circ f_2\left(\phi_2\left( D_{p_{1}}^{-} \right) \right) = g_2 \circ \phi_0^{-1} \circ \phi_0 \circ \phi_1^{-1} \circ \phi_1 \circ \phi_2^{-1} \left( f_2\left(\phi_2\left( D_{p_{1}}^{-} \right) \right)  \right) \subset f_0^{-n_2+1}\left( D_{p_{n_2}}^{--} \right).
\end{equation}

\noindent Moreover, $\dist( f_2( f_2(\phi_1( D_{p_{1}}^{-} ) ) , \partial f_0^{-n_2+1}( D_{p_{n_2}}^{--} ) ) > C_{n_2}$ by the expanding properties of $g_2$ in $S^+$. The argument that the relation (\ref{critical_point_for_f_1}) still holds with $f_1$ replaced by $f_2$ and $v$ any critical value of $f_2$ associated to $g_2|_{D_{p_{n_2}}}$ is proven analogously.

This procedure inductively defines the sequences $\boldsymbol{w}(p_{n_k})$, $\boldsymbol{\delta}(p_{n_k})$. For $k\geq1$, let the functions $g_k$, $\phi_k$, $f_k:=g_k\circ\phi_k^{-1}$ correspond to parameters

\begin{equation}\label{k^th_choice} \left(\boldsymbol{w}(p_{n_j})\right)_{j=1}^{k}\textrm{,   } \left(\boldsymbol{\delta}(p_{n_j})\right)_{j=1}^{k}\textrm{, and } \boldsymbol{w}(l)=\boldsymbol{\delta}(l)=0 \textrm{ for } l\not\in\left( p_{n_j} \right)_{j=1}^{k}. \end{equation}

\noindent From the above arguments it follows that

\begin{equation}\label{k^th_choice_argument} f_k^{n_{l+1}}\left(\phi_k\left( D_{p_{n_l}}^{-} \right)  \right) \subset  D_{p_{n_{l+1}}}^{--}\textrm{ and } f_k^{n_{l+1}}(v)\not\in D^{++}_{p_{n_{l+1}}} \textrm{ for } 1\leq l \leq k, \end{equation}

\noindent and any critical value $v$ of $f_l$ associated to $g_l|_{D_{p_{n_l}}}$.

Let $g_\infty$, $\phi_\infty$, $f_\infty:=g_\infty\circ\phi_\infty^{-1}$ correspond to parameters

\begin{equation}\label{last_choice} \left(\boldsymbol{w}(p_{n_j})\right)_{j=1}^{\infty}\textrm{,   } \left(\boldsymbol{\delta}(p_{n_j})\right)_{j=1}^{\infty}\textrm{, and } \boldsymbol{w}(l)=\boldsymbol{\delta}(l)=0 \textrm{ for } l\not\in\left( p_{n_j} \right)_{j=1}^{\infty}. \end{equation}

\noindent Since $(\phi_k)_{\overline{z}}/(\phi_k)_{z}\rightarrow(\phi_\infty)_{\overline{z}}/(\phi_\infty)_{z}$ a.e. (see Remark \ref{changing_delta_w}), $\phi_k\rightarrow\phi_{\infty}$ uniformly on compact subsets of $\mathbb{C}$. Thus the relations 

\begin{equation}\label{infty_choice_argument} f_\infty^{n_{l+1}}\left(\phi_\infty\left( D_{p_{n_l}}^{-} \right)  \right) \subset  D_{p_{n_{l+1}}}^{--}\textrm{ and } f_\infty^{n_{l+1}}(v)\not\in D^{++}_{p_{n_{l+1}}} \textrm{ for } 1\leq l \leq \infty, \end{equation}

\noindent and any critical value $v$ of $f_\infty$ associated to $g_\infty|_{D_{p_{n_l}}}$ follow from (\ref{k^th_choice_argument}). We claim that $f=f_\infty$ satisfies the conclusion of Proposition \ref{final_prop}, namely that $\phi_\infty\left( D_{p_{1}}^{-} \right)$ is contained in a wandering domain for $f_\infty$. Indeed, $f_\infty^{n_2}\left(\phi_\infty\left( D_{p_{1}}^{-} \right)  \right) \subset  D_{p_{n_2}}^{--}$, whence $\phi_\infty(D_{p_{n_2}}^{--})\subset D_{p_{n_2}}^{-}$ by Definition \ref{mu_definition}, so that $f_\infty^{n_2+n_3}(\phi_\infty\left( D_{p_{1}}^{-} \right)) \subset D_{p_{n_3}}^{--}$, and so forth. Thus any subsequence of $(f^{n}|_{\phi_\infty\left( D_{p_{1}}^{-} \right)})_{n=1}^{\infty}$ is either bounded, or contains a further subsequence converging to the constant function $\infty$.

\end{proof}

\section{Verifying that the Wandering Domain is indeed Univalent}\label{univalence}

In Proposition \ref{final_prop}, we showed that there was some choice of parameters $\boldsymbol{\delta}$, $\boldsymbol{w}$ such that the associated entire function $f:=g\circ\phi^{-1}$ with $\lambda$ as in Remark \ref{choice_of_lambda} and $\boldsymbol{m}$ as in Proposition \ref{selection_of_r_m} had a wandering component containing $\phi(D_{p_1}^-)$. To finish the proof of Theorem \ref{mainthm}, it then only remains to show that the forward orbit of this wandering component is univalent, which will follow once we show that this forward orbit has empty intersection with the singular set. 

\begin{prop}\label{univalence_prop} Let $n\geq 0$, and consider the function $f:=g\circ\phi^{-1}$ as given in Proposition \ref{final_prop}. The map $f$ is univalent on the Fatou component containing the domain $f^n(\phi(D_{p_1}^-))$.  

\end{prop}

\begin{proof} We have shown that $\phi(D_{p_1}^-)$ is contained in a wandering component for the map $f$. Let $U_n$ be the Fatou component containing $f^n(\phi(D_{p_1}^-))$, and let $\textrm{Crit}(f)$ be the set of critical values of $f$. The Proposition will follow once it is shown that $\textrm{Crit}(f)\cap U_n=\emptyset$ for all $n\geq0$, since $f$ has no asymptotic values by Theorem \ref{g_extension}. 



Let $v'$ be a critical value of $f$ associated to $g|_{D_{p_1}}$. We first show that $v'\not\in U_1$. Suppose by way of contradiction that $v'\in U_1$, or equivalently that $v:=f^{n_2}(v') \in f^{n_2}(U_1)=:U$.  Note that by (\ref{infty_choice_argument}), $v \not \in D_{p_{n_2}}^{++}$ whereas $u:=\phi(z_{p_{n_2}})\in U $. Without loss of generality, we may assume that $v$ lies in an $R$-component neighboring $D_{p_{n_2}}$. We call this $R$-component $V$. Since $\mathbb{R}$ is in the escaping set of $f$ (see \cite{Laz}), and $f\in\mathcal{B}$, it follows that $\mathbb{R} \subset\mathcal{J}(f)$ (see \cite{EL92}). Thus, $f^n(U)\subset\mathbb{H}:=\{z\in\mathbb{C} : \textrm{Im}(z)>0\}$ for all $n\in\mathbb{N}$. Denote by $d_{\mathbb{H}}, d_{f^n(U)}, d_{U}$ the hyperbolic distance in $\mathbb{H}, f^n(U), U$ respectively. We have then that


\begin{equation} \label{schwarz}
d_{\mathbb{H}}(f^n(u), f^n(v)) \leq d_{f^n(U)}(f^n(u), f^n(v)) \leq d_{U}(u,v) \textrm{ for all } n\in\mathbb{N},
\end{equation}

\noindent where both inequalities are consequences of Schwarz's lemma. We argue that the left hand side of (\ref{schwarz}) must tend to infinity with $n$, at least in some subsequence. This will be our needed contradiction.

We note that 

\[ |f(u)-1/2|=|g(z_{p_{n_2}})-1/2|=|f_0^{-n_2}(z_{p_{n_2}})-1/2|<1/4 \]

\noindent by Corollary \ref{dist_to_1/2}. We claim that $|f(v)|>1$. Indeed, consider a partition of $\mathbb{H}_r$ into $W_k:=\{z : \pi k<\textrm{Im}(z)<(k+1)\pi, \textrm{Re}(z)>0\}$ over $k\in\mathbb{Z}$, as illustrated to the left of Figure \ref{fig:tau}.
In the course of the proof of Theorem \ref{folding}, the map $\tau:V\rightarrow\mathbb{H}_r$ was replaced by a quasiconformal $\eta=:\hat{\tau}$, and the graph $T$ was decorated with edges to form a graph $T'$, so that $\hat{\tau}:V\setminus T'\rightarrow\mathbb{H}_r$ was quasiconformal and edges of $T'$ were sent to the edges $\{k<\textrm{Im}(z)<(k+1)\pi, \textrm{Re}(z)=0\}$ (see  \cite{Bis15} for details). Pulling the regions $W_k:=\{k<\textrm{Im}(z)<(k+1)\pi\}$ back under $\hat{\tau}$ we have regions $\hat{\tau}^{-1}(W_k)$. Now if $\hat\tau^{-1}(W_k)$ neighbors a $D$-component, recall that for $z\in\hat\tau^{-1}(W_k)$, we have $g(z)=\sigma\left(\hat{\tau}(z)\right)$ where $\sigma=\exp$ on $W_k$. In particular this means that as long as $z\in\hat\tau^{-1}(W_k)$, where $\hat\tau^{-1}(W_k)$ neighbors a $D$-component, then $|g(\hat{\tau}(z))|>1$.

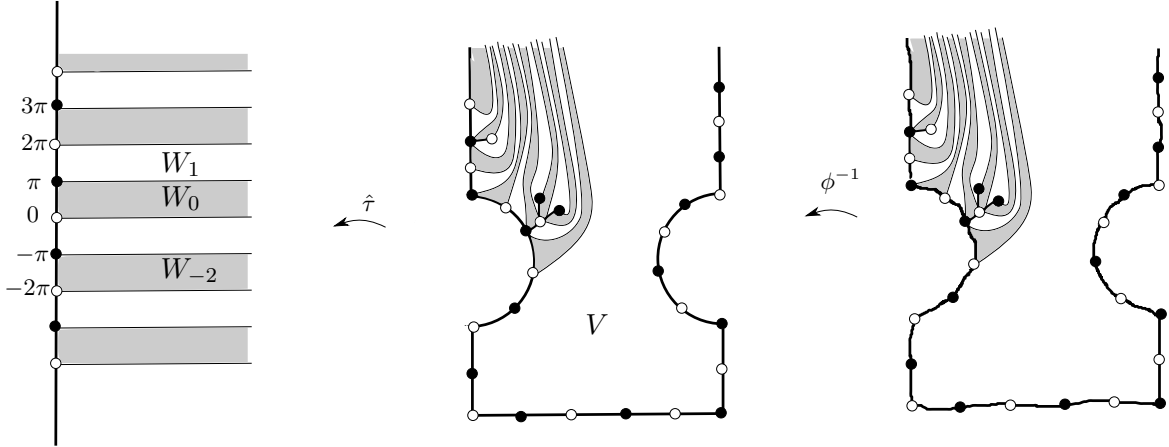
\begin{figure}[!htbp]
\centering
\setlength{\unitlength}{0.9\textwidth}
\includegraphics[width=0.9\textwidth]{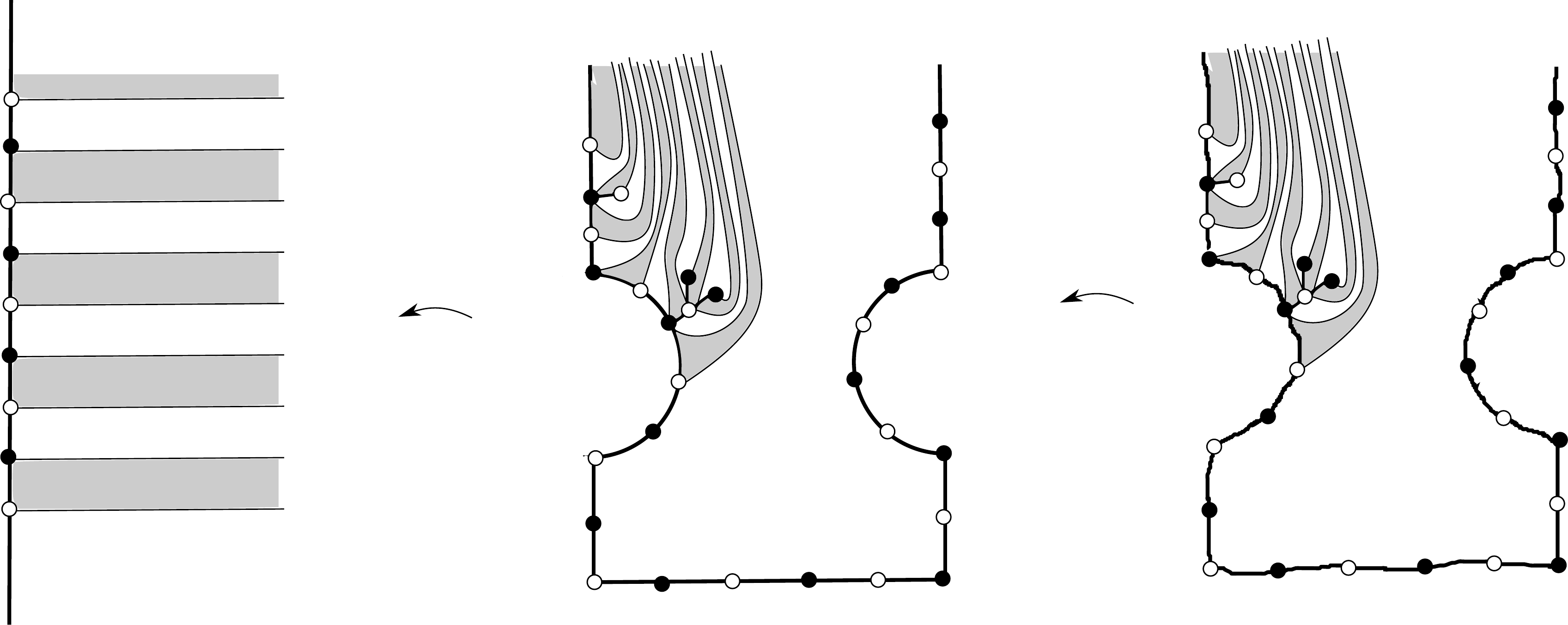}
\put(-1.025,0.298){\scs $3\pi$}
\put(-1.025,0.265){\scs $2\pi$}
\put(-0.9,0.245){\small $W_1$}
\put(-1.02,0.23){\scs $\pi$}
\put(-0.9,0.215){\small $W_0$}
\put(-1.02,0.2){\scs $0$}
\put(-1.03,0.165){\scs $-\pi$}
\put(-0.9,0.15){\small $W_{-2}$}
\put(-1.04,0.135){\scs $-2\pi$}
\put(-0.52,0.1){\small $V$}
\put(-0.72,0.21){\scs $\hat{\tau}$}
\put(-0.31,0.23){\scs $\phi^{-1}$}
\caption{\small Rough sketch of the preimages of the semistrips $W_k$ under $\hat{\tau}$ and $\phi^{-1}$. For simplicity, only a few preimages are shown.}
\label{fig:tau}
\end{figure}


Hence, we just need to argue that $\phi^{-1}(v) \in\hat\tau^{-1}(W_k)$, or $v \in(\hat\tau\circ\phi^{-1})^{-1}(W_k)$, where $\hat\tau^{-1}(W_k)$ neighbors a $D$-component (recall $f=\sigma\circ\hat{\tau}\circ\phi^{-1}$ in $V$). Indeed, $v\in U$ by assumption and $\phi^{-1}(v)\not\in D^{+}_{p_{n_2}}$ by Definition \ref{mu_definition} and since $v\not\in D^{++}_{p_{n_2}}$. Hence $v\in\phi(V)$ or $v\in\phi(\tilde{V})$ where $\tilde{V}$ is the other R-component neighboring $ D^{++}_{p_{n_2}}$, and we may assume without loss of generality that $v\in\phi(V)$. Moreover, $U$ can not intersect a region $(\hat\tau\circ\phi^{-1})^{-1}(W_\ell)\subset \phi(V)$ which does not neighbor $\phi(D_{p_{n_2}})$, the reason being that $U$ also contains points inside $\phi(D_{p_{n_2}})$ and therefore $U$ would have to cross the boundary of some region $(\hat{\tau}\circ\phi^{-1})^{-1}\left(W_l\right)$. Namely $U$ would have to cross $(\hat\tau\circ\phi^{-1})^{-1}(\{y=(l+1)\pi \})\cup(\hat\tau\circ\phi^{-1})^{-1}(\{y=l\pi\})$, but $(\hat\tau\circ\phi^{-1})^{-1}(\{y=(l+1)\pi \})\cup(\hat\tau\circ\phi^{-1})^{-1}(\{y=l\pi\})\subset\mathcal{J}(f)$, since they are sent to $\mathbb{R}$ by $f$, and this is a contradiction. Thus we have established that $|f(v)|>1$.

We have already noted that $|f(u)-1/2|<1/4$. By the expanding properties of the exponential, upon subsequent applications of $f$ we have that the distance between $f(u), f(v)$ increases upon each iterate. Indeed $f^k(u), f^k(v)$ must both lie in $S^+$ for $1\leq k<n_2$ (otherwise the wandering domain would have to cross $\partial  S^+ \subset J(f)$), and we have that $f(z)=\sigma(\lambda\sinh(\phi^{-1}(z))$ in $S^+$ (see (\ref{quasiregular_definition})). Note, for example, that $|f^{n_2}(u)-f^{n_2}(v)|>|\exp^{n_2}(3/4)-\exp^{n_2}(1)|$, and that $|\exp^{n}(3/4)-\exp^{n}(1)|\rightarrow\infty$ as $n\rightarrow\infty$. Next note that we know from the construction that $f^{n_2}(u)\in D^{--}_{p_{n_2}}$. This means that $|f^{n_2+1}(u)-1/2|<1/4$. What about $f^{n_2+1}(v)$? Well as observed in \cite{Laz}, by considering preimages of $\mathbb{R}$ one sees that there is a ray belonging to $\mathcal{J}(f)$ connecting $z_{n}-i$ to $\infty$ over all $n$ (remember $z_n$ is the center of the $D$-component $D_n$). This means that $f^{n_2}(v)$ must lie in an $R$-component neighboring $D_{p_{n_2}}$. Hence, as before, we know that $|f^{n_2+1}(v)|>1$.

Now we start over, only that we have $n_3>n_2$ more iterations of $f(z)=\sigma(\lambda\sinh(\phi^{-1}(z))$ in $S^+$ before the points $f^{n_2+1}(v),f^{n_2+1}(u)$ leave $S^+$. Hence, again the expanding properties of $\exp$, together with the fact that $|f^{n_2+1}(v)|>1, |f^{n_2+1}(u)-1/2|<1/4$, imply that $|f^{n_2+n_3}(v)-f^{n_2+n_3}(u)|>|\exp^{n_3}(3/4)-\exp^{n_3}(1)|$. Hence since $n_k\rightarrow\infty$ as $k\rightarrow\infty$, we know that $|f^{n_2+n_3+...+n_k}(v)-f^{n_2+n_3+...+n_k}(u)|\rightarrow\infty$ as $k\rightarrow\infty$. Moreover since by construction we know that $f^{n_2+n_3+...+n_k}(u)\in D_{p_{n_k}}$, we have that $d_{\mathbb{H}}(f^{n_2+n_3+...+n_k}(u), f^{n_2+n_3+...+n_k}(v))\rightarrow\infty$, which is our needed contradiction. Thus we conclude that $f^{n_2}(v')=v\not\in U=f^{n_2}(U_1)$, and so $v'\not\in U_1$. It remains to show that $v'\not\in U_n$ for $n\geq2$.

 Note that $v'\in\mathbb{D}$ as $g(D_{p_1})=\mathbb{D}$, whereas for $2\leq k \leq n_2$, $f^{k}(U_1)\cap\mathbb{D}=\emptyset$. We argue that $v'\not\in f^{n_2+1}(U_1)$, where we note that $f^{n_2+1}(U_1)$ is the Fatou component containing $f^{-n_3}(D_{p_{n_3}}^-)$. The same argument that shows $v\not\in U$ shows that the  Fatou component containing $f^{-n_3}(D_{p_{n_3}}^-)$ is contained in the disc centered at $ f^{-n_3}(z_{p_{n_3}})$ of radius $C\cdot|(f^{-n_3})'(g^{n_3}(1/2))|(1+2\mu_{p_{n_3}})$, where $C\sim1$ is a constant determined by Theorem \ref{thm:Koebe}. That $v'$ is not contained in this disc follows from:
 
\begin{align}\hspace{5mm} |v'-f^{-n_3}(z_{p_{n_3}})| \geq 
\\ \hspace{5mm} |v' - f^{-n_3}(g^{n_3}(1/2))| - | f^{-n_3}(g^{n_3}(1/2)) - f^{-n_3}(z_{p_{n_3}})| \geq  \nonumber
\\ \hspace{5mm}  |f^{-n_2}(z_{p_{n_2}}) - f^{-n_3}(g^{n_3}(1/2))| - |v'-f^{-n_2}(z_{p_{n_2}})| - | f^{-n_3}(g^{n_3}(1/2)) - f^{-n_3}(z_{p_{n_3}})| \geq \nonumber
\\ \dist(f^{-n_2}(z_{p_{n_2}}),\mathbb{R}) - |v'-f^{-n_2}(z_{p_{n_2}})| - | f^{-n_3}(g^{n_3}(1/2)) - f^{-n_3}(z_{p_{n_3}})| \geq \nonumber
\\ \textrm{ (Theorem \ref{thm:Koebe}) } \hspace{5mm} C' \left( \pi\left|  (f^{-n_2})'(g^{n_2}(1/2))  \right| - \left|  (f^{-n_2})'(g^{n_2}(1/2))  \right| - \pi\left|  (f^{-n_3})'(g^{n_3}(1/2))  \right| \right)\gg  \nonumber
\\ C|(f^{-n_3})'(g^{n_3}(1/2))|(1+2\mu_{p_{n_3}}).   \nonumber
\end{align}

\noindent where the first two inequalities follow from the triangle inequality. Similar arguments show that for $k>n_2+1$, $v'\not\in f^k(U_1)$. Thus the critical values of $f$ associated to $g|_{D_{p_1}}$ have empty intersection with the forward orbit of the wandering component containing $\phi(D_{p_1})$. A similar argument shows that the critical values of $f$ associated to $g|_{D_{p_{n_k}}}$ for $k>1$ have empty intersection with the forward orbit of this wandering component. The only remaining critical values by Theorem \ref{g_extension} are $0$, $\pm1$, which escape to infinity under $f$.








\end{proof}

\begin{rem} The argument given above is a refinement of that given in \cite{Laz} to show that the wandering component of \cite{Bis15} is a bounded subset of the plane. Indeed, the proof of Proposition \ref{univalence_prop} shows that the Fatou component containing $\phi(D_{p_1}^-)$ is contained in $D(z_{p_1}, 1+2\mu_{p_1})$. \end{rem}

\nocite{*}
\bibliographystyle{alpha}
\newcommand{\etalchar}[1]{$^{#1}$}

\end{document}